\documentclass[11pt]{article}
\pdfoutput=1 
\usepackage[top=1.3in,bottom=1.2in,left=1in,right=1in]{geometry}
\usepackage{amsmath}
\usepackage{amssymb}
\usepackage{amsthm}
\usepackage{mathtools}
\usepackage[only,llbracket,rrbracket]{stmaryrd}
\usepackage[scr=boondoxo]{mathalfa} 
\usepackage[english]{babel}
\usepackage{times}
\usepackage{graphicx}
\usepackage[font=small,figurewithin=section]{caption}
\usepackage{subcaption}
\usepackage{float}
\usepackage[numbers]{natbib} 
\usepackage{hyperref}
\usepackage[all]{xy}
\usepackage{multirow}
\usepackage{array}
\usepackage{enumitem}
\usepackage{tikz}
\usetikzlibrary{matrix,arrows,decorations.pathmorphing}
\usepackage{color}
\usepackage{authblk}

\numberwithin{equation}{section}
\makeatletter
\g@addto@macro\bfseries{\boldmath}
\makeatother

\hypersetup{colorlinks=true,citecolor=black,filecolor=black,linkcolor=black,urlcolor=black}
\hypersetup{pdfstartview=FitB,pdfpagemode=UseNone}

\setlist{nolistsep}

\newcolumntype{L}[1]{>{\raggedright\let\newline\\\arraybackslash\hspace{0pt}}m{#1}}
\newcolumntype{C}[1]{>{\centering\let\newline\\\arraybackslash\hspace{0pt}}m{#1}}
\newcolumntype{R}[1]{>{\raggedleft\let\newline\\\arraybackslash\hspace{0pt}}m{#1}}
\newcolumntype{N}{@{}m{0pt}@{}}

\SetSymbolFont{stmry}{bold}{U}{stmry}{m}{n}

\newcommand{\R}{\mathbb{R}}
\newcommand{\N}{\mathbb{N}}
\newcommand{\C}{\mathsf{C}}
\newcommand{\A}{\mathsf{S}}
\newcommand{\bfH}{\boldsymbol{H}}
\newcommand{\bfL}{\boldsymbol{L}}
\newcommand{\Sym}{\mathbb{S}}
\newcommand{\Skw}{\mathbb{A}}
\newcommand{\Mat}{\mathbb{M}}

\newcommand{\mesh}{\mathcal{T}}

\newcommand{\Efunc}{\mathcal{E}}

\newcommand{\mD}{\mathcal{D}}
\newcommand{\dd}{\,\mathrm{d}}

\newcommand{\T}{\mathsf{T}}
\newcommand{\sfN}{\mathsf{N}}
\newcommand{\sfR}{\mathsf{R}}

\newcommand{\bdry}{\partial}
\newcommand{\tr}{\mathrm{tr}}

\newcommand{\opt}{\mathrm{opt}}
\newcommand{\enr}{\mathrm{enr}}
\newcommand{\res}{\mathrm{res}}

\DeclareMathOperator*{\argmin}{arg\,min}
\DeclareMathOperator{\grad}{grad}
\DeclareMathOperator{\curl}{curl}
\let\div\relax 
\DeclareMathOperator{\div}{div}
\newcommand{\DIV}{\mathrm{div}}

\newtheoremstyle{boldremark}
    {\dimexpr\topsep/2\relax} 
    {\dimexpr\topsep/2\relax} 
    {}          
    {}          
    {\bfseries} 
    {.}         
    {.5em}      
    {}          
\theoremstyle{definition}
\newtheorem*{definition*}{Definition}
\theoremstyle{plain}
\newtheorem{theorem}{Theorem}
\numberwithin{theorem}{section}
\newtheorem{lemma}{Lemma}
\numberwithin{lemma}{section}
\newtheorem{corollary}{Corollary}
\numberwithin{corollary}{section}
\theoremstyle{boldremark}
\newtheorem{remark}{Remark}
\numberwithin{remark}{section}

\makeatletter
\newcommand{\dashto}[1][2pt]{
	\settowidth{\@tempdima}{${}\rightarrow{}$}
	\makebox[\@tempdima]{${}\rightarrow{}$}
  \makebox[-\@tempdima]{\hspace{-0.1\@tempdima}\color{white}\rule[0.5ex]{#1}{1pt}}
  \makebox[\@tempdima]{}
	}
\makeatother

\let\tilde\widetilde
\newcommand{\jump}[1]{\llbracket #1 \rrbracket}

\DeclareRobustCommand{\widefrac}[3][5pt]{%
  \frac{\hspace{#1}#2\hspace{#1}}{\hspace{#1}#3\hspace{#1}}}

\newcommand{\scS}{\mathscr{S}}
\newcommand{\scU}{\mathscr{U}}
\newcommand{\scD}{\mathscr{D}}
\newcommand{\scM}{\mathscr{M}}
\newcommand{\scP}{\mathscr{P}}
\newcommand{\scV}{\mathscr{V}}
\newcommand{\scW}{\mathscr{W}}

\newcommand{\fkn}{\mathfrak{n}}
\newcommand{\fku}{\mathfrak{u}}
\newcommand{\fkv}{\mathfrak{v}}
\newcommand{\fkw}{\mathfrak{w}}

\begin{document}

\title{The DPG methodology applied to different variational formulations of linear elasticity}
\author[*]{Brendan Keith}
\author[*]{Federico Fuentes}
\author[*]{Leszek Demkowicz}
\affil[*]{The Institute for Computational Engineering and Sciences (ICES), The University of Texas at Austin, 201 E 24th St, Austin, TX 78712-1229, USA}
\date{}

\maketitle

\renewcommand{\abstractname}{\large Abstract}
\begin{abstract}
\small 
The flexibility of the DPG methodology is exposed by solving the linear elasticity equations under different variational formulations, including some with non-symmetric functional settings (different infinite-dimensional trial and test spaces).
The family of formulations presented are proved to be mutually ill or well-posed when using traditional energy spaces on the whole domain.
Moreover, they are shown to remain well-posed when using broken energy spaces and interface variables.
Four variational formulations are solved in 3D using the DPG methodology.
Numerical evidence is given for both smooth and singular solutions and the expected convergence rates are observed.
\end{abstract}

\pagenumbering{arabic}
\section{Introduction}
\label{sec:Introduction}

In this paper we demonstrate the fitness of the DPG finite element method with optimal test spaces on various variational formulations of the nondimensionalized equations of linear elasticity,
\begin{equation}
 \label{eq:LinElast}
 \begin{alignedat}{3}
  -\div(\C: \varepsilon(u)) &= f\,,\quad &&\text{in }\, \Omega\,,\\
  u &= u_0\,, \quad &&\text{on }\, \Gamma_0\,,\\
  (\C: \varepsilon(u))\cdot\fkn &= g\,, \quad &&\text{on }\, \Gamma_1\,.
 \end{alignedat}
\end{equation}
We take $\Omega$ to be a simply connected smooth domain in $\R^3$ and let $\Gamma_0$ and $\Gamma_1$ be a partition of the boundary, $\overline{\Gamma_0\cup\Gamma_1} = \bdry\Omega$ with outward unit normal, $\fkn$.
Here, $u$ is the displacement, $\varepsilon(u) = \frac{1}{2}(\nabla u + \nabla u^\T)$ is the associated strain, $f$ is the body force, $g$ is the traction,\footnote{If $\overline{\Gamma_1} = \bdry \Omega$, then $f$ and $g$ must satisfy Signorini's compatibility condition $\int_\Omega f\cdot v \dd\Omega + \int_{\Gamma_1} g\cdot v \dd\Gamma = 0$ for all infinitesimal rigid displacements, $v$.} and $u_0$ is the prescribed displacement.
Meanwhile, 
$\C:\Sym\to\Sym$, is the elasticiy or stiffness tensor, where $\Sym$ denotes all symmetric $3\times3$ matrices.
For isotropic materials, it is expressed as $\C_{ijkl} = \lambda \delta_{ij}\delta_{kl} + \mu (\delta_{ik}\delta_{jl}+\delta_{il}\delta_{jk})$,
where $\lambda$ and $\mu$ are the Lam\'e parameters.

It can be shown that the standard Bubnov-Galerkin finite element method for linear elasticity computes the unique minimizer of the energy functional $\Efunc_1(v)=\int_\Omega (\frac{1}{2}\varepsilon(v):\C:\varepsilon(v)-f\cdot v)\dd\Omega -\int_{\Gamma_1}g\cdot v \dd\Gamma$, over all candidates, $v$, in a discrete space of displacements, $U_h$.
By proceeding from an energy minimization we guarantee to compute the best possible solution (measured in the energy) allowed in our set of computable solutions (trial space).
In this sense, the formulation has the obvious desirable quality that there is a meaningful metric of solution relevance as measured by the energy functional.
The typical approach in commercial software is to use exactly the standard Bubnov-Galerkin variational formulation to simulate and predict elastic behavior in materials.

Notwithstanding the above method, there are important circumstances where such a simple energy minimization principle is avoided.
Another prominent variational formulation for linear elasticity is the well-known mixed method \cite{brezzi2012mixed}. These discretizations stem from the minimax problem on the Hellinger-Reissner energy functional $\Efunc_2(\tau,v) = -\int_\Omega \big(\frac{1}{2}\tau : \C^{-1} : \tau + \DIV\tau\cdot v + f\cdot v \big)\dd\Omega + \int_{\Gamma_0}u_0\cdot(\tau\!\cdot\!\fkn) \dd\Gamma$ \cite{mixedelas3d}, an energy principle equivalent to minimization of $\Efunc_1$ \cite{Ciarlet13,Ekeland1976}.
Here $v$ is a displacement variable and $\tau=\tau^\T$ is a stress variable.
Such a formulation results in a discretization which avoids volumetric locking and also guarantees a locally conservative stress tensor \cite{brezzi2012mixed}.
Of course, this formulation also guarantees a best possible solution although it is measured in a different way and the trial spaces differ.

Likewise, other energy principles exist for linear elasticity problems.
In fact, just for this single problem a total of $14$ complementary-dual energy principles are presented in \cite{Oden76}, each leading to a different variational formulation.
Some may not be easily amenable to computation but perspective should be given that there is little to regard as sacred or more physical about one formulation over another.
Ultimately, whatever the physical principle (energy functional) employed, the equations of linear elasticity are ubiquitous; beyond their functional setting, they do not change even though they can be derived in different ways and posed over different spaces.
In principle, at the infinite-dimensional level the solution will always be the same but at the computational level the differences can become very important.

In the DPG method, we do not make a quandary over the best physical principle to employ for our choice of optimality.
Instead, without access to the exact solution outright, we seek the best numerical solution available to us once the trial space and variational formulation are set.
This is achieved by considering a minimization problem on the residual of the discrete solution taken through a user-defined norm in the discrete test space.
The ramifications of this methodology are substantial, however analyzing most of them are not the particular focus of this paper. Instead, we intend only to demonstrate the utility of the methodology on various variational formulations.
We will now outline some of the history and recent developments of DPG.

The optimal stability DPG methodology \cite{demkowicz2010class,demkowicz2011class}, referred here simply as ``DPG'', was originally envisioned as a practical Petrov-Galerkin finite element method which would naturally transfer the stability of the infinite dimensional problem onto the discrete system.
This is achieved by exploiting a natural isometry between a Hilbert space and its dual, called the Riesz map, and the ability to localize its computation by using broken test spaces.
In a difficult problem, instead of tuning stability parameters as is commonplace in standard stabilized methods, the DPG method algorithmically approximates an \textit{optimal test space} to a tuneable accuracy in a way that applies to all well-posed variational problems.
The tuning parameter in the DPG method is usually the order of a local test space called the \textit{enriched test space} where the Riesz map (in the user-defined norm) is computed.
The larger this parameter, the more accurate the approximation of the optimal test space. 
For every computation in this paper we found it sufficient to choose an enriched test space one order larger than the trial space.
Using a larger enrichment may not be a great hindrance, because the feasibility of the method is offered from the fact that all computations on this higher order enriched space are localized.
Therefore, provided the the element-local computations have been distributed (which can be done in parallel) and are made efficiently, the computational cost of the method is essentially independent of the enrichment parameter.
However, each element-local computation can sometimes be computationally intensive if the enrichment parameter is too high.
In this context, the choice of the user-defined norm of the test space can play a fundamental role in efficiently obtaining a well-approximated optimal test space while only requiring a modest enrichment parameter.

DPG distinguishes the trial and test space differently from the outset and because of this trait it is applicable to often neglected, non-symmetric variational formulations.
This originally led to the DPG method with ultraweak variational formulations, a formulation wherein the trial space is naturally discontinuous.
Some highlights of ultraweak variational formulations are given in \cite{roberts2014dpg,DemkowiczGopalakrishnan13_2,Carstensen15,bui2013unified}. 
Indeed, in this setting, DPG has largely been applied to singular perturbation problems and other problems in computational mechanics where stability is difficult to achieve such as advection diffusion \cite{chan2014robust} and thin-body problems \cite{Niemi2011}.
Recently, DPG has been applied in the context of space-time problems in \cite{Ellis2015}.

Usually DPG operates with a discontinuous test space.
However, the trial space must be \textit{globally} conforming and for this reason, it is somewhat unique among discontinuous Galerkin methods \cite{bui2011relation,demkowicz2013primal}.
Indeed, in this paper we intend to emphasize that DPG is not limited only to ultraweak variational formulations.
In fact, we will show that a reformulation of a variational problem over a broken test space can be seen as a way of embedding the original variational problem into a larger one.
We then show that this new variational problem over broken test spaces is well-posed if and only if the original unbroken problem is well-posed.
Therein, because DPG inherits its stability from the underlying infinite dimensional problem, we always guarantee convergence of the method, provided the enrichment parameter is large enough.



In the context of linear elasticity, the DPG ultraweak setting has been applied to 2D problems in \cite{Bramwell12} resulting in two different methods, one of which has been complemented by a complete error analysis in \cite{gopalakrishnan2014analysis}.
There is also work in 2D elasticity for low order methods \cite{HellwigThesis}.
In this article we contribute to the previous DPG ventures in linear elasticity by implementing the method for the 3D equations in four different variational formulations.
We also apply a newly developed theory for broken variational formulations \cite{Carstensen15}, which we use to prove their stability.
Lastly, we include what we believe is the contemporary observation that all of the variational formulations which we have considered are mutually well or ill-posed (a similar assertion has been proved in the context of Maxwell equations in \cite{Carstensen15}).
This is important because it avoids having to present an independent proof of well-posedness for each different variational formulation.



\subsection{Outline}



In Section \ref{sec:LinearElasticity} we propose five variational formulations for linear elasticity.
These equations arise naturally by formal integration by parts of a first order system equivalent to \eqref{eq:LinElast}.
The formulations are observed to be mutually ill or well-posed.

In Section \ref{sec:BrokenSpaces} we define the broken energy spaces along with necessary interface (or broken trace) spaces. 
Using these spaces, we derive the associated five variational formulations in the \textit{broken} setting.
We close this section by demonstrating the well-posedness of each of these formulations upon the assumption that the respective ``unbroken'' formulations in Section~\ref{sec:LinearElasticity} are also well-posed.

In Section \ref{sec:MinResidual} we demonstrate how each of these formulations fit into the DPG framework and elaborate upon the specifics of the DPG methodology, including the computation of the residual to use in adaptivity. 

Finally, in Section \ref{sec:Numerics} we present our numerical experiments of the DPG method with four variational formulations applied to 3D smooth and singular linear elasticity problems.



\section{Linear elasticity and some variational formulations}
\label{sec:LinearElasticity}

\subsection{Energy spaces}

As a prelude to the variational formulations defined in this section, we must first describe the functional spaces where the trial and test variables lie.
These are typically known as energy spaces.
First, we define the most basic underlying energy spaces and norms for a domain $\Omega$,
\begin{equation}
	\begin{alignedat}{5}
		L^2(\Omega)&\!=\!\{u:\Omega\to\R\mid\|u\|_{L^2(\Omega)}<\infty\}\,\,&&\text{with}\,\,
			&\|u\|_{L^2(\Omega)}^2&&&\!=\!\textstyle{\int_{\Omega}|u|^2\dd\Omega}\,,\\
		H^1(\Omega)&\!=\!\{u:\Omega\to\R\mid\|u\|_{H^1(\Omega)}<\infty\}\,\,&&\text{with}\,\,
			&\|u\|_{H^1(\Omega)}^2&&&\!=\!\|u\|_{L^2(\Omega)}^2+\textstyle{\int_{\Omega}|\nabla u|^2\dd\Omega}\,,\\
		H(\div,\Omega)&\!=\!\{v:\Omega\to\R^3\mid\|v\|_{H(\div,\Omega)}<\infty\}\,\,&&\text{with}\,\,
			&\|v\|_{H(\div,\Omega)}^2&&&\!=\!\textstyle{\int_{\Omega}|v|^2\dd\Omega}+\|\div(v)\|_{L^2(\Omega)}^2\,.
	\end{alignedat}
\end{equation}
Here, the functions are defined up to sets of measure zero, and $|\cdot|$ is the standard Euclidean norm.
Note the expressions here and throughout this work are dimensionally consistent because all physical variables and constants are assumed to have been nondimensionalized.

Next we define some of the vector and matrix energy spaces we will be using,
\begin{equation}
	\begin{aligned}
		\bfL^2(\Omega)&=\{u:\Omega\to\R^3\mid u_i\in L^2(\Omega),i=1,2,3\}\,,\\
		\bfL^2(\Omega;\mathbb{U})&=\{u:\Omega\to\mathbb{U}\mid u_{ij}\in L^2(\Omega),i,j=1,2,3\}\,,\\
		\bfH^1(\Omega)&=\{u:\Omega\to\R^3\mid u_i\in H^1(\Omega),i=1,2,3\}\,,\\
		\bfH(\div,\Omega)&=\{\sigma:\Omega\to\Mat\mid (\sigma_{i1},\sigma_{i2},\sigma_{i3})\in H(\div,\Omega),i=1,2,3\}\,,\\
	\end{aligned}
	\label{eq:vectormatrixenergyspaces}
\end{equation}
where $\mathbb{U}$ is a subspace of $\Mat$, the space of $3\times3$ matrices.
In particular $\mathbb{U}$ can be the symmetric matrices, $\Sym$, the antisymmetric (or skew symmetric) matrices, $\Skw$, or $\Mat$ itself.
The Hilbert norms are naturally defined from the underlying energy spaces.
That is, $\|u\|_{\bfL^2(\Omega)}^2=\sum_{i=1}^3\|u_i\|_{L^2(\Omega)}^2$, etc.
Finally it is useful to define certain subspaces which vanish at some parts of the boundary, $\bdry\Omega$, which is assumed to be Lipschitz.
Namely,
\begin{equation}
	\begin{aligned}
		\bfH_{\Gamma_0}^1(\Omega)&=\{u\in\bfH^1(\Omega)\mid u_i|_{\Gamma_0}=0,i=1,2,3\}\,,\\
		\bfH_{\Gamma_1}(\div,\Omega)&=\{\sigma\in\bfH(\div,\Omega)\mid
					(\sigma_{i1},\sigma_{i2},\sigma_{i3})|_{\Gamma_1}\cdot \fkn_{\Gamma_1}=0,i=1,2,3\}\,,\\
		\bfH_{\Gamma_1}(\div,\Omega;\Sym)&=\{\sigma:\Omega\to\Sym\mid \sigma\in\bfH_{\Gamma_1}(\div,\Omega)\}\,,
	\end{aligned}
\end{equation}
where $\Gamma_0$ and $\Gamma_1$ are relatively open subsets of $\bdry\Omega$ satisfying $\overline{\Gamma_0\cup\Gamma_1}=\bdry\Omega$ and $\Gamma_0\cap\Gamma_1=\varnothing$ and where $\fkn_{\Gamma_1}$ is the unit exterior normal to $\Omega$ along $\Gamma_1$.

Naturally all the energy spaces in \eqref{eq:vectormatrixenergyspaces} can be defined on a domain different from $\Omega$, such as on an arbitrary element $K\subseteq\Omega$.
In fact, when it is clear from the context, the domain is absorbed into the notation for convenience.

\subsection{First order systems}

The equations of linear elasticity can be derived, as was previously mentioned, from energy principles, but in reality, they represent a linear approximation of a nonlinear operator which is naturally expressed as a first order system.
Per convention, the linearization is done in the reference configuration about a zero displacement at which the stress is assumed to vanish (zero residual stress).
This first order system consists of two equations,
\begin{equation}
 \label{eq:LinElastSystem}
 -\div(\C:\varepsilon(u)) = f\quad \text{in }\, \Omega \qquad \iff \qquad
 \left\{
  \begin{aligned}
   \sigma - \C: \varepsilon(u)  & = 0\quad &&\text{in }\, \Omega\,,\\
   -\div\sigma &= f\quad &&\text{in }\, \Omega\,.
  \end{aligned}
 \right.
\end{equation}
The first equation is a linearization of the original constitutive law and relates the Cauchy stress tensor, $\sigma$, to the engineering strain tensor, $\varepsilon(u)$.
We note that this equation may be rewritten as
\begin{equation}
	\A:\sigma - \varepsilon(u) = 0\,,
	\label{eq:RobustConstitutiveLaw}
\end{equation}
where $\A=\C^{-1}:\Sym\to\Sym$, the inverse of $\C$ over symmetric matrices, is known as the compliance tensor.
For isotropic materials it is $\A_{ijkl}=\frac{1}{4\mu}(\delta_{ik}\delta_{jl}+\delta_{il}\delta_{jk})-\frac{\lambda}{2\mu(3\lambda+2\mu)}\delta_{ij}\delta_{kl}$
The second equation is the conservation of linear momentum in the reference configuration (with the understanding that the first Piola-Kirchhoff stress tensor is equal to the Cauchy stress tensor up to a small error in this standard linearized setting).
Conservation of angular momentum is contained implicitly in the assumption that $\sigma = \sigma^\T$.

\subsection{Variational equations}
\label{sec:variationalequations}

If we assume that $f\in \bfL^2(\Omega)$, the conservation law is equivalent to the variational equation
\begin{equation}
\label{eq:TrivialLinearMomentum}
 -\int_\Omega \div\sigma \cdot v \dd\Omega= \int_\Omega f\cdot v \dd\Omega \,, \quad \text{for all }\, v\in \bfL^2(\Omega)\,.
\end{equation}
Due to the symmetry of the stress tensor, $\sigma=\sigma^\T$, it is natural to consider $\sigma\in\tilde{g}+\bfH_{\Gamma_1}(\div,\Omega;\Sym)$, where $\tilde{g}\in\bfH(\div,\Omega;\Sym)$ is an extension of the prescribed boundary traction $g$ from $\Gamma_1$ to $\Omega$.
However, in practice, the space $\bfH_{\Gamma_1}(\div,\Omega;\Sym)$ is very difficult to discretize \cite{elas3dfamily,mixedelas3d,qiu2011mixed}.
Instead it is often assumed $\sigma\in\tilde{g}+\bfH_{\Gamma_1}(\div,\Omega)$, with the symmetry of $\sigma$ being imposed weakly through the extra equation,
\begin{equation}
	\label{eq:SigmaWeakSymmetry}
	\int_\Omega \sigma:w\dd\Omega = 0 \,, \quad \text{for all }\, w\in \bfL^2(\Omega;\Skw)\,,
\end{equation}
and where $\tilde{g}\in\bfH(\div,\Omega)$ is now a possibly different extension of $g$ from $\Gamma_1$ to $\Omega$.

Formally integrating \eqref{eq:TrivialLinearMomentum} by parts, an equation closely related to the principle of virtual work is obtained,
\begin{equation}
\label{eq:PplVirtualWork}
	\int_\Omega \sigma : \nabla v \dd\Omega = \int_\Omega f\cdot v \dd\Omega + \int_{\Gamma_1} g\cdot v \dd\Gamma\,,
		\quad \text{for all }\, v\in \bfH^1_{\Gamma_0}(\Omega)\,.
\end{equation}
Here, to enforce the symmetry it makes sense to take $\sigma\in\bfL^2(\Omega;\Sym)$ which is easy to discretize.
Note that $v\in \bfH^1_{\Gamma_0}(\Omega)$ in \eqref{eq:PplVirtualWork}, while $v\in\bfL^2(\Omega)$ in \eqref{eq:TrivialLinearMomentum}.

Likewise, after testing with $\tau$, the constitutive law in \eqref{eq:LinElastSystem} may be written in a variational form as
\begin{equation}
\label{eq:TrivialConstitutiveLaw}
	\int_\Omega \sigma:\tau\dd\Omega-\int_\Omega\nabla u:\C:\tau\dd\Omega= 0\,,
		\quad\text{for all }\,\tau\in \bfL^2(\Omega;\Sym) \,,
\end{equation}
where it was used $\varepsilon(u):\C=\nabla u:\C$, with the domain of $\C:\Sym\to\Sym$ being extended naturally to $\C:\Mat\to\Sym$ (i.e., $\C|_{\Skw}=0$).
Here, due to the presence of $\nabla u$, it makes sense to have $u\in\tilde{u}_0+\bfH_{\Gamma_0}^1(\Omega)$, where $\tilde{u}_0\in\bfH^1(\Omega)$ is an extension of the prescribed boundary displacement $u_0$ from $\Gamma_0$ to $\Omega$.

To get an alternate variational form of the constitutive equation it is more convenient to consider the characterization provided in \eqref{eq:RobustConstitutiveLaw}.
This equation is easier to integrate by parts and avoids volumetric locking in the limit of incompressible materials due to the use of the compliance tensor, $\A$, which is robust with respect to the Lam\'{e} parameters (in the sense that $\|\C\|\to\infty$ while $\|\A\|<\infty$ as $\lambda\to\infty$).
A first attempt at integrating this relation by parts with a symmetric $\tau=\tau^\T$ yields the expression $\div\tau$, meaning that one should take $\tau\in\bfH_{\Gamma_1}(\div,\Omega;\Sym)$.
This revives the difficulties of discretizing $\bfH_{\Gamma_1}(\div,\Omega;\Sym)$.
To overcome the issue, one must introduce an extra solution variable called the infinitesimal rotation tensor, $\omega$, which satisfies
\begin{equation}
	\nabla u=\varepsilon(u)+\omega\qquad\Rightarrow\qquad\A:\sigma - \nabla u + \omega = 0\,.
	\label{eq:RobustConstitutiveLawII}
\end{equation}
Testing and integrating by parts then yields
\begin{equation}
	\label{eq:RelaxedConstitutiveLaw}
 	\int_\Omega\sigma:\A:\tau\dd\Omega+\int_\Omega\omega:\tau\dd\Omega+\int_\Omega u\cdot\div\tau\dd\Omega
 		=\int_{\Gamma_0}u_0\cdot\tau n\dd\Gamma\,,\quad \text{for all }\,\tau\in \bfH_{\Gamma_1}(\div,\Omega)\,,
\end{equation}
where the domain of $\A$ is extended trivially from $\Sym$ to $\Mat$ (i.e., $\A|_\Skw=0$).
Here, it is natural to consider $u\in\bfL^2(\Omega)$ and $\omega\in\bfL^2(\Omega;\Skw)$, which are both easy to discretize.

\subsection{4+1 variational formulations} 
\label{sec:variational_formulations}

First, for sake of exposition, throughout the rest of this work we assume that the displacement and traction boundary conditions are homogeneous, meaning $u_0 = 0$ and $g = 0$ (so one can choose extensions $\tilde{u}_0=0$ and $\tilde{g}=0$).
As we have just demonstrated, the first order system of equations of linear elasticity can be posed in their weak form in a variety of ways.
Indeed, by simply making different choices about integrating by parts we can arrive at the following four variational formulations for linear elasticity.

\textbf{$(\scS)$ Strong formulation}
\begin{flalign}
\label{eq:TrivialFormulation}
 \left\{
  \begin{alignedat}{3}
		\multicolumn{1}{l}{$\text{Find } u\in \bfH^1_{\Gamma_0}(\Omega),\, \sigma\in \bfH_{\Gamma_1}(\div,\Omega),$}\hspace{-200pt} & &&\\
		\int_\Omega \sigma : \tau \dd\Omega - \int_\Omega \nabla u : \C : \tau \dd\Omega &= 0
			\,, \quad &&\text{for all }\, \tau \in \bfL^2(\Omega;\Sym) \,,\\
		-\int_\Omega \div\sigma \cdot v \dd\Omega &= \int_\Omega f\cdot v \dd\Omega \,, \quad &&\text{for all }\, v\in \bfL^2(\Omega)\,,\\
		\int_\Omega \sigma:w\dd\Omega &= 0 \,, \quad &&\text{for all }\, w\in \bfL^2(\Omega;\Skw)\,.
  \end{alignedat}
 \right. &&
\end{flalign}

\textbf{$(\scU)$ Ultraweak formulation}
\begin{flalign}
\label{eq:UWeakFormulation}
 \left\{
  \begin{alignedat}{3}
   	\multicolumn{1}{l}{$\text{Find } u\in \bfL^2(\Omega),\, \sigma\in \bfL^2(\Omega;\Sym),
   		\,\omega\in\bfL^2(\Omega;\Skw),$}\hspace{-200pt} & &&\\
   \int_\Omega\sigma:\A:\tau\dd\Omega+\int_\Omega\omega:\tau\dd\Omega+\int_\Omega u\cdot\div\tau\dd\Omega
   	&= 
   		0\,,\quad &&\text{for all }\, \tau \in \bfH_{\Gamma_1}(\div,\Omega)\,,\\
   \int_\Omega \sigma : \nabla v \dd\Omega
   	&= \int_\Omega f\cdot v \dd\Omega
   		\,, \quad &&\text{for all }\, v\in \bfH^1_{\Gamma_0}(\Omega)\,.
  \end{alignedat}
 \right. &&
\end{flalign}

\textbf{$(\scD)$ Dual Mixed formulation}
\begin{flalign}
\label{eq:Mixed1Formulation}
 \left\{
  \begin{alignedat}{3}
		\multicolumn{1}{l}{$\text{Find } u\in \bfH^1_{\Gamma_0}(\Omega),\,\sigma\in\bfL^2(\Omega;\Sym),$}\hspace{-200pt} & &&\\
   	\int_\Omega \sigma : \tau \dd\Omega - \int_\Omega \nabla u : \C : \tau \dd\Omega &= 0
   		\,, \quad &&\text{for all }\, \tau \in \bfL^2(\Omega;\Sym) \,,\\
   \int_\Omega \sigma : \nabla v \dd\Omega
   	&= \int_\Omega f\cdot v \dd\Omega
   		\,, \quad &&\text{for all }\, v\in \bfH^1_{\Gamma_0}(\Omega)\,.
  \end{alignedat}
 \right. &&
\end{flalign}

\textbf{$(\scM)$ Mixed formulation}
\begin{flalign}
\label{eq:Mixed2Formulation}
 \left\{
  \begin{alignedat}{3}
  	\multicolumn{1}{l}{$\text{Find } u\in \bfL^2(\Omega),\, \sigma\in \bfH_{\Gamma_1}(\div,\Omega),
   		\,\omega\in\bfL^2(\Omega;\Skw),$}\hspace{-200pt} & &&\\
   \int_\Omega\sigma:\A:\tau\dd\Omega+\int_\Omega\omega:\tau\dd\Omega+\int_\Omega u\cdot\div\tau\dd\Omega
   	&= 
   		0\,,\quad &&\text{for all }\, \tau \in \bfH_{\Gamma_1}(\div,\Omega)\,,\\
   	-\int_\Omega \div\sigma \cdot v \dd\Omega &= \int_\Omega f\cdot v \dd\Omega \,, \quad &&\text{for all }\, v\in \bfL^2(\Omega)\,,\\
   	\int_\Omega \sigma:w\dd\Omega &= 0 \,, \quad &&\text{for all }\, w\in \bfL^2(\Omega;\Skw)\,.
  \end{alignedat}
 \right. &&
\end{flalign}


Observe that the dual mixed formulation can also be rewritten in second order form.
Therefore, we allow for one more variational formulation, equivalent to \eqref{eq:Mixed1Formulation},

\textbf{$(\scP)$ Primal formulation}
\begin{flalign}
\label{eq:PrimalFormulation}
 \left\{
  \begin{aligned}
   \multicolumn{1}{l}{$\text{Find } u\in \bfH^1_{\Gamma_0}(\Omega),$}\hspace{-200pt} & \\
   \int_\Omega \nabla u : \C : \nabla v \dd\Omega &= \int_\Omega f\cdot v \dd\Omega
   \,, \quad \text{for all }\, v\in \bfH^1_{\Gamma_0}(\Omega)\,.
  \end{aligned}
 \right. &&
\end{flalign}

This final variational formulation is easily the most pervasive in the finite element literature.
Treated with the Bubnov-Galerkin method and conforming finite elements, it has the strong advantage of computational efficiency for it involves only a single solution variable and so the required degrees of freedom in computation are usually significantly less than in each other discretization.
This formulation produces a symmetric coercive bilinear form and so also a symmetric positive definite stiffness matrix.
Moreover, when using piecewise-linear isoparametric elements, it is well known to always reproduce infinitesimal rigid displacements exactly.
Compared to the dual mixed formulation, the primal formulation is essentially superior in all ways since even the stress tensor, $\sigma$, can be computed a posteriori from the $\bfH^1(\Omega)$ solution variable. 
For this reason, we avoid computing with the dual mixed formulation in Section \ref{sec:Numerics}.

The mixed formulation is also well known in the literature for it avoids volumetric locking in nearly incompressible scenarios (as $\lambda\to\infty$) and globally preserves the conservation law $-\int_\Omega\div\sigma\dd\Omega=\int_\Omega f \dd\Omega$ in the discrete solution.
The law holds element-wise as well as long as the indicator function of each element is in the test function space.
Neither of these traits are present in the primal formulation.
Moreover, the mixed method can also be discretized with conforming finite elements with the Bubnov-Galerkin method.
As with the primal method, this is due to the fact that the test and trial spaces are the same.
The mixed formulation can be simplified when using the symmetric space $\bfH_{\Gamma_1}(\div,\Omega;\Sym)$ for $\sigma$ and $\tau$.
However, stable finite element spaces satisfying strong symmetry in the stress variable are very difficult to produce.
Some notable treatments of these difficulties are considered in \cite{elas3dfamily,mixedelas3d,SchoeberlSinwel11,PechsteinSchoberl12}.

The strong formulation can be recast as the first order least squares finite element method (see Section \ref{sec:hybrid_dpg}).
This method is easy to implement and always produces positive definite stiffness matrices.

The lesser studied ultraweak formulation is not often used because it does not immediately admit a stable discretization due to the test and trial spaces being different.
This formulation has been the traditional setting for applying the DPG methodology.
With this methodology, we will show that the formulation indeed can admit a stable discretization. 
Amongst many advantages is that it is volumetric locking-free \cite{Bramwell12}.\footnote{
In fact, since $\|\A\|<\infty$ as $\lambda\to\infty$, all four initial formulations (but not primal) can be recast in a volumetric locking-free robustly stable form by using the compliance tensor, $\A$, instead of the stiffness tensor, $\C$.
Hence, using \eqref{eq:RobustConstitutiveLaw} one can obtain a replacement to \eqref{eq:TrivialConstitutiveLaw}.
Namely, $\int_\Omega \sigma:\A:\tau\dd\Omega-\int_\Omega\nabla u:\tau\dd\Omega=0$ for all $\tau\in\bfL^2(\Omega;\Sym)$.}


Of course, other formulations of linear elasticity are also possible such as those derived from the Hu-Washizu principle \cite{Oden76}.
Notably, volumetric locking can also be avoided by introducing a pressure term, but at the cost making traction (normal stress) boundary conditions more difficult to handle \cite{HughesFamous}.



\subsection{Well-posedness} 
\label{sec:wellposedconforming}

One of the main results is stated in the next theorem, whose proof is relegated to Appendix \ref{app:WellPosedness}.

\begin{theorem}
\label{thm:mutuallywellposed}
The variational formulations $(\scS)$, $(\scU)$, $(\scD)$, $(\scM)$ and $(\scP)$ are mutually ill or well-posed.
That is, if any single formulation is well-posed, then all others are also well-posed.
\end{theorem}

It is well known that the primal variational formulation, $(\scP)$, is well-posed by using Korn's inequality \cite{Ciarlet13} whenever $\Gamma_0\neq\varnothing$. 
Hence, the following corollary immediately follows.

\begin{corollary}
\label{cor:allwellposed}
Let $\,\Gamma_0$ be relatively open in $\bdry\Omega$.
If $\,\Gamma_0\neq\varnothing$, then the variational formulations $(\scS)$, $(\scU)$, $(\scD)$, $(\scM)$ and $(\scP)$ are well-posed.
\end{corollary}

\section{Variational formulations with broken test spaces}
\label{sec:BrokenSpaces}

As we will see later, to compute optimal test functions it is necessary to invert the test space Riesz operator which can be an expensive procedure when the test spaces are globally conforming.
However, if the test spaces are broken with respect to a mesh, this inversion becomes a local procedure which can be completed independently with respect to each element.
Moving to broken tests spaces in a variational formulation comes at the cost of introducing new interface variables along the skeleton of the mesh and therefore involves more unknowns. It can be considered as a way of embedding the original formulation into a larger one.
Consistently, this results in well-posed ``broken'' variational formulations whose solutions correspond to the solutions of the original formulations in a way made precise by Theorem~\ref{thm:BrokenWellposedness}.

The majority of the material in this section is developed in greater detail in \cite{Carstensen15}.
Here, we repeat some relevant results from this larger theory which will be necessary for our treatment of linear elasticity.

\subsection{Broken energy spaces} 
\label{sec:broken_test_spaces}

We now assume that the domain $\Omega$ is partitioned into a mesh of elements, $\mesh$, and we assume that each element in the mesh, $K\in\mesh$, has a Lipschitz continuous boundary, $\bdry K$, like all polytopal elements do.

A broken energy space is a mesh dependent test space having no continuity constraints across mesh element interfaces.
The ones we will be most interested in are defined as
\begin{equation}
	\begin{aligned}
		\bfL^2(\mesh)&=\{u\in\bfL^2(\Omega)\mid \forall K\in\mesh,u|_K\in\bfL^2(K)\}=\bfL^2(\Omega)\,,\\
		\bfH^1(\mesh)&=\{u\in\bfL^2(\Omega)\mid \forall K\in\mesh,u|_K\in\bfH^1(K)\}\,,\\
		\bfH(\div,\mesh)&=\{\sigma\in\bfL^2(\Omega;\Mat)\mid \forall K\in\mesh,\sigma|_K\in\bfH(\div,K)\}\,,
	\end{aligned}
\end{equation}
and their respective norms are defined naturally as
\begin{equation}
\label{eq:BrokenNorms}
	\|u\|_{\bfL^2(\mesh)} \!=\! \|u\|_{\bfL^2(\Omega)}\,,\quad
	\|u\|_{\bfH^1(\mesh)}^2 \!=\! \sum_{K\in\mesh} \|u|_K\|_{\bfH^1(K)}^2\,,\quad
	\|\sigma\|_{\bfH(\div,\mesh)}^2 \!=\! \sum_{K\in\mesh} \|\sigma|_K\|_{\bfH(\div,K)}^2\,.
\end{equation}
Similar definitions hold for $\bfL^2(\mesh;\mathbb{U})=\bfL^2(\Omega;\mathbb{U})$ for each $\mathbb{U}\subseteq\Mat$ that we have previously considered.
Moreover, we use the notation
\begin{equation}
	(\cdot,\cdot)_{\mesh}=\sum_{K\in\mesh}(\cdot,\cdot)_{K}\,,
\end{equation}
where for any $K\subseteq\Omega$, $(\cdot,\cdot)_K$ is either $(u,v)_K=\int_K u\cdot v\dd K$ if $u,v\in\bfL^2(K)$, or $(\sigma,\tau)_K=\int_K \sigma:\tau\dd K$ if $\sigma,\tau\in\bfL^2(K;\Mat)$.

Note that these broken spaces are essentially copies of the underlying energy space at the local element level.
As such, it is easier to construct a discretization for them than their ``unbroken'' counterparts because the requirement for global conformity of the basis functions at the interelement boundaries has been removed.

\begin{remark}
One can easily see that the broken energy norms are legitimate due to the fact that the underlying norm is \textit{localizable}.
That is, it is a map dependent on some open subset $K\subseteq\Omega$, which is a well-defined norm for any $K\subseteq\Omega$.
Indeed, $\|\cdot\|_{\bfH^1(K)}$, $\|\cdot\|_{\bfH(\div,K)}$ and $\|\cdot\|_{\bfL^2(K)}$ are norms for any $K\subseteq\Omega$.
An important limitation to this construction is the $\bfH^1(K)$ seminorm on the space $\bfH^1_{\Gamma_0}(\Omega)$ for $\Gamma_0\neq\varnothing$, which cannot be extended in the same way.
In fact, $|\cdot|_{\bfH^1(K)}=\int_K|\nabla(\cdot)|^2\dd K$ is a norm for $K=\Omega$, but is not a norm if $K\subseteq\Omega$ does not share part of its boundary with $\Gamma_0$ (i.e. $\bdry K\cap\Gamma_0=\varnothing$).
Of course, one is free to choose problem dependent norms for the test spaces (such as graph norms), and it can be extremely advantageous to do so, but for simplicity, we do not motivate any exotic norms in this work.
\end{remark}

\subsection{Interface spaces} 
\label{sec:interface_spaces}

The interface variables to be introduced lie in interface spaces induced by well-known surjective local element trace operators defined as
\begin{equation}
\label{eq:ElementTrace}
	\begin{alignedat}{5}
		\tr_{\grad}^K&:\bfH^1(K)\to\bfH^{\frac{1}{2}}(\bdry K)\,,&&\qquad&\tr_{\grad}^K u&&&=u|_{\bdry K}\,,\\
		\tr_{\div}^K&:\bfH(\div,K)\to\bfH^{-\frac{1}{2}}(\bdry K)\,,&&\qquad&\tr_{\div}^K \sigma&&&=\sigma|_{\bdry K}\cdot \fkn_{\bdry K}\,.
	\end{alignedat}
\end{equation}
Here, $\fkn_{\bdry K}$ denotes the unit outward normal on $\bdry K$ and the contraction $\sigma|_{\bdry K}\cdot \fkn_{\bdry K}$ is considered along the second index (i.e., row-wise).
The local trace operators are continuous and the spaces $\bfH^{\frac{1}{2}}(\bdry K)$ and $\bfH^{-\frac{1}{2}}(\bdry K)$ are (topologically) dual to each other when they are suited with minimum energy extension norms.

The next step is to determine the mesh trace operators, which are defined as
\begin{equation}
	\begin{alignedat}{5}
		\tr_{\grad}&:\bfH^1(\mesh)\to\prod_{K\in\mesh}\bfH^{\frac{1}{2}}(\bdry K)\,,&&\qquad
			&\tr_{\grad} u&&&=\prod_{K\in\mesh}\tr_{\grad}^K u\,,\\
		\tr_{\div}&:\bfH(\div,\mesh)\to\prod_{K\in\mesh}\bfH^{-\frac{1}{2}}(\bdry K)\,,&&\qquad
			&\tr_{\div}\sigma&&&=\prod_{K\in\mesh}\tr_{\div}^K \sigma\,.
	\end{alignedat}
\end{equation}
From these, we inherit the relevant interface spaces,
\begin{equation}
	\label{eq:InterfaceSpaces}
	\begin{alignedat}{5}
		\bfH^{\frac{1}{2}}(\bdry\mesh)&=\tr_{\grad}(\bfH^1(\Omega))\,,&&\qquad
			&\bfH^{\frac{1}{2}}_{\Gamma_0}(\bdry\mesh)&&&=\tr_{\grad}(\bfH^1_{\Gamma_0}(\Omega))\subseteq\bfH^{\frac{1}{2}}(\bdry\mesh)\,,\\
		\bfH^{-\frac{1}{2}}(\bdry\mesh)&=\tr_{\div}(\bfH(\div,\Omega))\,,&&\qquad
			&\bfH^{-\frac{1}{2}}_{\Gamma_1}(\bdry\mesh)&&&=\tr_{\div}(\bfH_{\Gamma_1}(\div,\Omega))\subseteq\bfH^{-\frac{1}{2}}(\bdry\mesh)\,,
	\end{alignedat}
\end{equation}
which are endowed with the minimum energy extension norms of $\bfH^1(\Omega)$ and $\bfH(\div,\Omega)$ respectively. 
In \cite{Carstensen15} these norms are importantly shown to be equal to
\begin{equation}
	\label{eq:InterfaceNorms}
	\begin{aligned}
		\|\hat{u}\|_{\bfH^{\frac{1}{2}}(\bdry\mesh)}
			&=\sup_{\sigma\in\bfH(\div,\mesh)\setminus\{0\}}\widefrac[-12pt]{|\langle\hat{u},\tr_{\div}\sigma\rangle_{\bdry\mesh}|}
				{\qquad\quad\|\sigma\|_{\bfH(\div,\mesh)}}\quad\,,\\
		\|\hat{\sigma}_\fkn\|_{\bfH^{-\frac{1}{2}}(\bdry\mesh)}
			&=\sup_{u\in\bfH^1(\mesh)\setminus\{0\}}\widefrac[0pt]{|\langle\hat{\sigma}_\fkn,\tr_{\grad}u\rangle_{\bdry\mesh}|}
				{\qquad\|u\|_{\bfH^1(\mesh)}}\,,
	\end{aligned}
\end{equation}
for all $\hat{u}\in\bfH^{\frac{1}{2}}(\bdry\mesh)$ and $\hat{\sigma}_\fkn\in\bfH^{-\frac{1}{2}}(\bdry\mesh)$.
Here,
\begin{equation}
	\langle\cdot,\cdot\rangle_{\bdry\mesh}=\sum_{K\in\mesh}\langle\cdot,\cdot\rangle_{\bdry K}\,,
	\label{eq:meshinnerproduct}
\end{equation}
with $\langle\cdot,\cdot\rangle_{\bdry K}$ being the duality pairing $\langle\cdot,\cdot\rangle_{\bfH^{\frac{1}{2}}(\bdry K)\times\bfH^{-\frac{1}{2}}(\bdry K)}$ or $\langle\cdot,\cdot\rangle_{\bfH^{-\frac{1}{2}}(\bdry K)\times\bfH^{\frac{1}{2}}(\bdry K)}$ depending upon the context.

Notice that $\bfH^{\frac{1}{2}}(\bdry\mesh)=\tr_{\grad}(\bfH^1(\Omega))\subsetneq\tr_{\grad}(\bfH^1(\mesh))$.
Indeed, elements in $\tr_{\grad}(\bfH^1(\mesh))$ intuitively may have different values at the two sides of the inner facets of the mesh.
Similar assertions hold for $\bfH^{-\frac{1}{2}}(\bdry\mesh)\subsetneq\tr_{\div}(\bfH(\div,\mesh))$.
These observations are confirmed with aid of the following remark and lemma.

\begin{remark}
Let $\mathcal{G}$ be the set containing all the unrepeated facets of the elements of the mesh.
Facets can be in the interior of the mesh in which case they are shared by two elements alone, say $K^+$ and $K^-$, inducing opposite normal vectors $\fkn^+$ and $\fkn^-=-\fkn^+$, or they can be on the exterior in which case they are part of a single element $K^+$ and have a unique outward normal $\fkn^+$.
For each facet $F\in\mathcal{G}$ a normal is selected and fixed. Thus, for interior facets a normal is chosen between  $\fkn^+$ and $\fkn^-$, say $\fkn^+$ is always chosen, while for exterior facets we can only choose $\fkn^+$.
For any piecewise \textit{smooth} $v\in\bfH^1(\mesh)$ and $\tau\in\bfH(\div,\mesh)$, define $v^\pm=v|_{K^\pm}$ and $\tau^\pm=\tau|_{K^\pm}$ with $v^-=0$ and $\tau^-=0$ whenever the facet is on the boundary.
Then, the facet traces are $\tr_{\grad}^{F}v=v^+|_{F}$ and $\tr_{\div}^{F}\tau=\tau^+|_{F}\cdot \fkn^+$, while the facet jumps are $\jump{\tr_{\grad}^{F}v}=v^+|_{F}-v^-|_{F}$ and $\jump{\tr_{\div}^{F}\tau}=(\tau^+|_{F}-\tau^-|_{F})\cdot \fkn^+$.
With these conventions, observe that
%
%
%
\begin{equation}
	\begin{aligned}
  	\langle\tr_{\div}\tau,\tr_{\grad}v\rangle_{\bdry\mesh}
  		&=\sum_{F\in\mathcal{G}}\langle\tr_{\div}^{F}\tau,\jump{\tr_{\grad}^{F}v}\rangle_{F}\,,
  			\quad\text{for \textit{smooth}}\quad \tau\in\bfH(\div,\Omega),\,v\in\bfH^1(\mesh)\,,\\
		\langle\tr_{\grad}v,\tr_{\div}\tau\rangle_{\bdry\mesh}
  		&=\sum_{F\in\mathcal{G}}\langle\tr_{\grad}^{F}v,\jump{\tr_{\div}^{F}\tau}\rangle_{F}\,,
  			\quad\text{for \textit{smooth}}\quad v\in\bfH^1(\Omega),\,\tau\in\bfH(\div,\mesh)\,,
  \end{aligned}
  \label{eq:smoothsumofjumps}
\end{equation}
where $\langle\cdot,\cdot\rangle_{F}$ is the $\bfL^2$ inner product on the face $F$.
Note that this inner product is well-defined for \textit{smooth} functions, but does not generalize to arbitrary elements in $\bfH^1(\Omega)$, $\bfH^1(\mesh)$, $\bfH(\div,\Omega)$ and $\bfH(\div,\mesh)$ because $\langle\cdot,\cdot\rangle_{F}$ will \textit{not} extend to a well-defined duality pairing.
However, the expression in \eqref{eq:meshinnerproduct} always holds because the duality pairings are well-defined on the full boundaries of the elements as opposed to just a particular facet in the mesh boundary.
In this context, \eqref{eq:smoothsumofjumps} suggests that $\langle\cdot,\cdot\rangle_{\bdry\mesh}$ can be interpreted as the sum against all jumps across element interfaces.
One would expect that if $\langle\cdot,\cdot\rangle_{\bdry\mesh}$ vanishes for a given broken trial function and many test functions, then all jumps are zero and the trial function is single-valued and lies in the underlying unbroken space.
Indeed, this is the content of the following lemma, which is proved in Appendix \ref{app:zerojump}.
\end{remark}

\begin{lemma}
\label{lem:CharacterizationOfTraces}
Let $\,\Gamma_0$ and $\,\Gamma_1$ be relatively open subsets in $\bdry\Omega$ satisfying $\overline{\Gamma_0\cup\Gamma_1}=\bdry\Omega$ and $\Gamma_0\cap\Gamma_1=\varnothing$.
\begin{enumerate}[font=\upshape,label={(\roman*)},ref={\thelemma(\roman*)}]
	\item Let $v\in\bfH^1(\mesh)$. Then $v\in\bfH^1_{\Gamma_0}(\Omega)$ if and only if $\langle\hat{\tau}_\fkn,\tr_{\grad}v\rangle_{\bdry\mesh}=0$ for all $\hat{\tau}_\fkn\in\bfH^{-\frac{1}{2}}_{\Gamma_1}(\bdry\mesh)$. \label{lem:H1Subspace}
	\item Let $\tau\in\bfH(\div,\mesh)$. Then $\tau\in\bfH_{\Gamma_1}(\div,\Omega)$ if and only if $\langle\hat{u},\tr_{\div}\tau\rangle_{\bdry\mesh}=0$ for all $\hat{u}\in\bfH^{\frac{1}{2}}_{\Gamma_0}(\bdry\mesh)$. \label{lem:HdivSubspace}
\end{enumerate}
%
\end{lemma}

\subsection{Broken variational formulations}
\label{sec:BrokenVariationalFormulations}

Variational formulations on broken test spaces can be derived from accumulating all of the contributions coming from element-wise integration across the mesh.
Throughout this section, we assume homogeneous displacement and traction boundary conditions, $u_0=0$ on $\Gamma_0$ and $g=0$ on $\Gamma_1$.

We proceed as in Section \ref{sec:variationalequations}.
Formally integrating over each element instead of the whole domain, we obtain each of the first order equations in \eqref{eq:LinElastSystem} in unrelaxed and relaxed variational forms with similar modifications to avoid discretizing the space $\bfH(\div,K;\Sym)$.

Choosing to avoid integration by parts, we can express the equations of linear elasticity as
\begin{flalign*}
 \left\{
  \begin{alignedat}{3}
   	\multicolumn{1}{l}{$\text{Find } u\in \bfH^1_{\Gamma_0}(\Omega),\,\sigma\in \bfH_{\Gamma_1}(\div,\Omega),
   		\text{ such that for each } K\in\mesh,$} \hspace{-300pt}& &&\\
   	(\sigma,\tau)_K-(\C:\nabla u ,\tau)_K&=0\,,
   		\quad &&\text{for all }\,\tau\in\bfL^2(K;\Sym) \,,\\
   	-(\div\sigma,v)_K&=(f,v)_K \,,
   		\quad &&\text{for all }\, v\in \bfL^2(K)\,,\\
   	(\sigma,w)_K&=0\,,
   		\quad &&\text{for all }\,w\in\bfL^2(K;\Skw) \,.
  \end{alignedat}
 \right. &&
\end{flalign*}
Therefore, immediately accumulating all of the single element contributions yields the first broken formulation. 

\textbf{$(\scS_\mesh)$ Strong formulation}
\begin{flalign}
\label{eq:TrivialBrokenFormulation}
 \left\{
  \begin{alignedat}{3}
  	\multicolumn{1}{l}{$\text{Find } u\in\bfH^1_{\Gamma_0}(\Omega),\,\sigma\in \bfH_{\Gamma_1}(\div,\Omega),$}\hspace{-300pt}& &&\\
   	(\sigma,\tau)_\mesh-(\C:\nabla u,\tau)_\mesh&=0\,,
   		\quad&&\text{for all }\,\tau\in\bfL^2(\mesh;\Sym)\!=\!\bfL^2(\Omega;\Sym) \,,\\
   	-(\div\sigma,v)_\mesh&=(f,v)_\mesh\,,
   		\quad&&\text{for all }\,v\in\bfL^2(\mesh)\!=\!\bfL^2(\Omega)\,,\\
   	(\sigma,w)_\mesh&=0\,,
   		\quad&&\text{for all }\,w\in\bfL^2(\mesh;\Skw)\!=\!\bfL^2(\Omega;\Skw)\,.
  \end{alignedat}
 \right. &&
\end{flalign}
Observe that this formulation is equivalent to the original strong formulation in \eqref{eq:TrivialFormulation}.


However, to obtain a broken ultraweak variational formulation analogous to \eqref{eq:UWeakFormulation}, a more elaborate analysis is required. 
Choosing to integrate by parts both of the equations at an element level, we obtain expressions akin to \eqref{eq:PplVirtualWork} and \eqref{eq:RelaxedConstitutiveLaw}, which our solution variable ostensibly satisfies,
\begin{flalign}
\label{eq:FormalUltraWeak}
 \left\{
  \begin{alignedat}{3}
  	(\A:\sigma,\tau)_\mesh\!+\!(\omega,\tau)_\mesh\!+\!(u,\div\tau)_\mesh\!-\!
  		\langle\tr_{\grad}u,\tr_{\div}\tau\rangle_{\bdry\mesh}&\!=\!0\,,
  			\,\,\, &&\text{for all }\, \tau\in\bfH_{\Gamma_1}(\div,\mesh)\,,\\
   	(\sigma,\nabla v)_\mesh\!-\!\langle\tr_{\div}\sigma,\tr_{\grad}v\rangle_{\bdry\mesh}&\!=\!(f,v)_\mesh\,,
   		\,\,\, &&\text{for all }\,v\in\bfH^1_{\Gamma_0}(\mesh)\,.
  \end{alignedat}
 \right. &&
\end{flalign}
Here, the legitimacy of $\tr_{\grad}u$ and $\tr_{\div}\sigma$ is not yet guaranteed as we have not specified the energy spaces for the trial variables.

Previously, in \eqref{eq:TrivialBrokenFormulation}, the notation $(\cdot,\cdot)_\mesh$ was awkward and we could have easily replaced this sum of element-wise inner products with $(\cdot,\cdot)_\Omega$. 
In the new expressions in \eqref{eq:FormalUltraWeak}, we insist on the notation $(\cdot,\cdot)_\mesh$ as the divergence and gradient operations, $\div$ and $\nabla$, are only defined acting upon the broken test space within element boundaries, \textit{not} over the entire domain, $\Omega$.


For the time being, let us reconsider testing only against $v_0\in\bfH^1_{\Gamma_0}(\Omega)$ and $\tau_0\in\bfH_{\Gamma_1}(\div,\Omega)$ which come only from subsets of the broken test spaces.
In this case, the boundary terms $\langle\tr_{\grad}u,\tr_{\div}\tau\rangle_{\bdry\mesh}$ and $\langle\tr_{\div}\sigma,\tr_{\grad}v\rangle_{\bdry\mesh}$ in \eqref{eq:FormalUltraWeak} are inclined to vanish if we recall Lemma~\ref{lem:CharacterizationOfTraces}. 
If this were so, we would then actually recover the original equations of the ultraweak variational formulation in \eqref{eq:UWeakFormulation},
\begin{flalign}
\label{eq:RecoveredUltraWeak}
 \left\{
  \begin{alignedat}{3}
   	(\A:\sigma,\tau_0)_\mesh+(\omega,\tau_0)_\mesh+(u,\div\tau_0)_\mesh&=0\,,
  			\quad &&\text{for all }\, \tau_0\in\bfH_{\Gamma_1}(\div,\Omega)\,,\\
   	(\sigma,\nabla v_0)_\mesh&=(f,v_0)_\mesh\,,
   		\quad &&\text{for all }\,v_0\in\bfH^1_{\Gamma_0}(\Omega)\,,
  \end{alignedat}
 \right. &&
\end{flalign}
Observing only \eqref{eq:RecoveredUltraWeak}, we are motivated to search for the trial variables in the same spaces as in the original well-posed formulation, $u\in\bfL^2(\Omega)$, $\sigma\in\bfL^2(\Omega;\Sym)$ and $\omega\in\bfL^2(\Omega;\Skw)$.
However, with that assumption, the discarded terms $\langle\tr_{\grad}u,\tr_{\div}\tau\rangle_{\bdry\mesh}$ and $\langle\tr_{\div}\sigma,\tr_{\grad}v\rangle_{\bdry\mesh}$ from \eqref{eq:FormalUltraWeak} would not be well-defined and Lemma~\ref{lem:CharacterizationOfTraces} would not apply to them.
To deal with this complication, we introduce new \textit{interface} variables to solve a complementary problem,
\begin{flalign}
\label{eq:ComplementaryUltraWeak}
 \left\{
  \begin{alignedat}{3}
   	\multicolumn{1}{l}{$\text{Let }u_0\in \bfL^2(\Omega),\,\sigma_0\in\bfL^2(\Omega;\Sym)\text{ and }\omega_0\in\bfL^2(\Omega;\Skw)
   		\text{ be the unique solution to \eqref{eq:RecoveredUltraWeak}.}$}\hspace{-400pt}\\
   	\multicolumn{1}{l}{$\text{Find }\hat{u}\in \bfH_{\Gamma_0}^{\frac{1}{2}}(\bdry\mesh),
   		 	\,\hat{\sigma}_\fkn\in\bfH_{\Gamma_1}^{-\frac{1}{2}}(\bdry\mesh),$} \hspace{-400pt}\\
   		\langle\hat{u},\tr_{\div}\tau\rangle_{\bdry\mesh}&=(\A:\sigma_0,\tau)_\mesh+(\omega_0,\tau)_\mesh+(u_0,\div\tau)_\mesh\,,
   			\quad &&\text{for all }\,\tau\in\bfH_{\Gamma_1}(\div,\mesh)\,,\\
  		\langle\hat{\sigma}_\fkn,\tr_{\grad}v\rangle_{\bdry\mesh}&=(\sigma_0,\nabla v)_\mesh-(f,v)_\mesh\,,
  			\quad &&\text{for all }\,v\in\bfH^1_{\Gamma_0}(\mesh)\,.
  \end{alignedat}
 \right. &&
\end{flalign}

We are now left with two problems to solve, \eqref{eq:RecoveredUltraWeak} and \eqref{eq:ComplementaryUltraWeak}, but fortunately these can be expressed as one single problem posed simultaneously, which is precisely the broken ultraweak formulation.

\textbf{$(\scU_\mesh)$ Ultraweak formulation}
\begin{flalign}
\label{eq:UWeakBrokenFormulation}
 \left\{
  \begin{alignedat}{3}
   	\multicolumn{1}{l}{$\text{Find } u\in\bfL^2(\Omega),\,\sigma\in\bfL^2(\Omega;\Sym),\,\omega\in\bfL^2(\Omega;\Skw),\,
   		\hat{u}\in \bfH_{\Gamma_0}^{\frac{1}{2}}(\bdry\mesh),\,
   			\hat{\sigma}_\fkn \in \bfH_{\Gamma_1}^{-\frac{1}{2}}(\bdry\mesh),$} \hspace{-400pt}\\
   	(\A:\sigma,\tau)_\mesh\!+\!(\omega,\tau)_\mesh\!+\!(u,\div\tau)_\mesh\!-\!\langle\hat{u},\tr_{\div}\tau\rangle_{\bdry\mesh}&=0\,,
  		\quad &&\text{for all }\,\tau\in\bfH_{\Gamma_1}(\div,\mesh)\,,\\
   	(\sigma,\nabla v)_\mesh-\langle\hat{\sigma}_\fkn,\tr_{\grad}v\rangle_{\bdry\mesh}&=(f,v)_\mesh\,,
   		\quad &&\text{for all }\,v\in\bfH^1_{\Gamma_0}(\mesh)\,.
  \end{alignedat}
 \right. &&
\end{flalign}


A similar process renders the rest of the broken variational formulations.

\textbf{$(\scD_\mesh)$ Dual Mixed formulation}
\begin{flalign}
\label{eq:Mixed1BrokenFormulation}
 \left\{
  \begin{alignedat}{3}
   	\multicolumn{1}{l}{$\text{Find } u\in \bfH^1_{\Gamma_0}(\Omega),\,\sigma\in \bfL^2(\Omega;\Sym),\,
   		\hat{\sigma}_\fkn \in\bfH_{\Gamma_1}^{-\frac{1}{2}}(\bdry\mesh),$} \hspace{-400pt}\\
   	(\sigma,\tau)_\mesh-(\C:\nabla u,\tau)_\mesh&=0\,,
   		\quad&&\text{for all }\,\tau\in\bfL^2(\mesh;\Sym)\!=\!\bfL^2(\Omega;\Sym) \,,\\
   	(\sigma,\nabla v)_\mesh-\langle\hat{\sigma}_\fkn,\tr_{\grad}v\rangle_{\bdry\mesh}&=(f,v)_\mesh\,,
   		\quad &&\text{for all }\,v\in\bfH^1_{\Gamma_0}(\mesh)\,.
  \end{alignedat}
 \right. &&
\end{flalign}

\textbf{$(\scM_\mesh)$ Mixed formulation}
\begin{flalign}
\label{eq:Mixed2BrokenFormulation}
 \left\{
  \begin{alignedat}{3}
   	\multicolumn{1}{l}{$\text{Find } u\in\bfL^2(\Omega),\,\sigma\in\bfH_{\Gamma_1}(\div,\Omega),\,\omega\in\bfL^2(\Omega;\Skw),\,
   		\hat{u}\in\bfH_{\Gamma_0}^{\frac{1}{2}}(\bdry\mesh),$} \hspace{-400pt}\\
   	(\A:\sigma,\tau)_\mesh\!+\!(\omega,\tau)_\mesh\!+\!(u,\div\tau)_\mesh\!-\!\langle\hat{u},\tr_{\div}\tau\rangle_{\bdry\mesh}&=0\,,
  		\,\,&&\text{for all }\,\tau\in\bfH_{\Gamma_1}(\div,\mesh)\,,\\
   	-(\div\sigma,v)_\mesh&=(f,v)_\mesh\,,
   		\,\,&&\text{for all }\,v\in\bfL^2(\mesh)\!=\!\bfL^2(\Omega)\,,\\
   	(\sigma,w)_\mesh&=0\,,
   		\,\,&&\text{for all }\,w\in\bfL^2(\mesh;\Skw)\!=\!\bfL^2(\Omega;\Skw)\,.\!\!\!\!\!\!\!\!\!\!\!\!\!\!\!
  \end{alignedat}
 \right. &&
\end{flalign}

\textbf{$(\scP_\mesh)$ Primal formulation}
\begin{flalign}
\label{eq:PrimalBrokenFormulation}
 \left\{
  \begin{aligned}
   	\multicolumn{1}{l}{$\text{Find } u\in\bfH^1_{\Gamma_0}(\Omega),\,
   		\hat{\sigma}_\fkn\in\bfH_{\Gamma_1}^{-\frac{1}{2}}(\bdry\mesh),$} \hspace{-400pt} \\
   	(\C:\nabla u ,\nabla v )_\mesh-\langle\hat{\sigma}_\fkn,\tr_{\grad}v\rangle_{\bdry\mesh}&=(f,v)_\mesh\,,
   		\quad\text{for all }\, v\in \bfH^1_{\Gamma_0}(\mesh)\,.
  \end{aligned}
 \right. &&
\end{flalign}

The philosophy we advocate for is to consider the original problem with unbroken test spaces as being embedded into a larger problem having a larger broken test space.
This section closes by demonstrating more formally that these new broken variational formulations are indeed well-posed problems and that the solution variables $u$ and $\sigma$ do agree with the solutions of the original unbroken formulations.

\subsection{Well-posedness of broken variational formulations} 
\label{sec:WellPosedness}


To show that the formulations proposed in Section \ref{sec:BrokenVariationalFormulations} are well-posed we will make use of the following theorem.
The interested reader can inspect the proof in \cite[Theorem 3.1]{Carstensen15}.

\begin{theorem}
\label{thm:BrokenWellposedness}
Let $U_0$, $\hat{U}$ and $V$ be Hilbert spaces over a fixed field $\mathbb{F}\in\{\R,\mathbb{C}\}$.
Let $\ell:V\to\mathbb{F}$ be a continuous linear form, and let $b_0:U_0\times V \to \mathbb{F}$ and $\hat{b}:\hat{U}\times V \to \mathbb{F}$ be continuous bilinear forms if $\,\mathbb{F}=\R$ or sesquilinear forms if $\,\mathbb{F}=\mathbb{C}$. 
With $U = U_0 \times \hat{U}$ and $\|\cdot\|_{U}^2=\|\cdot\|_{U_0}^2+\|\cdot\|_{\hat{U}}^2$, define $b:U\times V \to \mathbb{F}\,$ for all $(\fku_0,\hat{\fku})\in U$ and $\fkv\in V$ by
\begin{equation*}
	b((\fku_0,\hat{\fku}),\fkv) = b_0(\fku_0,\fkv) + \hat{b}(\hat{\fku},\fkv)\,,
\end{equation*}
and let
\begin{equation*}
	V_0 = \{\fkv\in V \mid \hat{b}(\hat{\fku},\fkv)=0 \,\text{ for all }\,\hat{\fku}\in\hat{U}\}\,.
\end{equation*}
Assume:
\begin{enumerate}[font=\upshape,label={($\gamma_0$)},after=\vspace{-0.5\baselineskip}]
	\item There exists $\gamma_0>0$ such that for all $\fku_0\in U_0$,
		\begin{equation*}
			\sup_{\fkv_0\in V_0\setminus\{0\}} \frac{|b_0(\fku_0,\fkv_0)|}{\|\fkv_0\|_V}\geq\gamma_0\|\fku_0\|_{U_0}\,.
		\end{equation*}
		\label{item:Assumption1}
\end{enumerate}
\begin{enumerate}[font=\upshape,label={($\hat{\gamma}$)},after=\vspace{-0.5\baselineskip}]
	\item There exists $\hat{\gamma}>0$ such that for all $\hat{\fku}\in \hat{U}$,
		\begin{equation*}
			\sup_{\fkv\in V\setminus\{0\}} \frac{|\hat{b}(\hat{\fku},\fkv)|}{\|\fkv\|_V}\geq\hat{\gamma}\|\hat{\fku}\|_{\hat{U}}\,.
		\end{equation*}
		\label{item:Assumption2}
\end{enumerate}
Then:
\begin{enumerate}[font=\upshape,label={($\gamma$)},after=\vspace{-0.0\baselineskip}]
	\item There exists $\gamma=(\frac{1}{\gamma_0^2}+\frac{1}{\hat{\gamma}^2}(\frac{M_0}{\gamma_0}+1))^{-\frac{1}{2}}>0$ such that for all $(\fku_0,\hat{\fku})\in U$,
		\begin{equation*}
			\sup_{\fkv\in V\setminus\{0\}} \frac{|b((\fku_0,\hat{\fku}),\fkv)|}{\|\fkv\|_V}\geq\gamma\|(\fku_0,\hat{\fku})\|_{U}\,,
		\end{equation*}
		\label{item:TotalInfSup}
		where $M_0\geq\|b_0\|=\sup_{(\fku_0,\fkv)\in U_0\times V\setminus\{(0,0)\}}\frac{|b_0(\fku_0,\fkv)|}{\|\fku_0\|_{U_0}\|\fkv\|_V}$.
\end{enumerate}
Moreover, if $\ell$ satisfies the compatibility condition,
\begin{equation*}
	\ell(\fkv)=0\,\text{ for all }\,\fkv\in V_{00}\,,
\end{equation*}
where
\begin{equation*}
	V_{00} = \{\fkv_0\in V_0 \mid b_0(\fku_0,\fkv_0)=0\,\text{ for all }\,\fku_0\in U_0\}\,,
\end{equation*}
which is always true if $V_{00}=\{0\}$, then the problem
\begin{flalign*}
\label{eq:AbstractBrokenFormulation}
\qquad
 \left\{
  \begin{aligned}
   &\text{Find }\,(\fku_0,\hat{\fku})\in U,\\
   & b((\fku_0,\hat{\fku}),\fkv) = \ell(\fkv)\,, \quad \text{for all } \,\fkv\in V\,,
  \end{aligned}
 \right. &&
\end{flalign*}
has a unique solution $(\fku_0,\hat{\fku})$ satisfying the estimate
\begin{equation*}
	\|(\fku_0,\hat{\fku})\|_{U}\leq\frac{1}{\gamma}\|\ell\|_{V'}\,.
\end{equation*}
Furthermore, the component $\fku_0$ from the unique solution is also the unique solution to the problem
\begin{flalign*}
\label{eq:AbstractRootFormulation}
\qquad
 \left\{
  \begin{aligned}
   &\text{Find }\,\fku_0\in U_0\,,\\
   & b_0(\fku_0,\fkv_0) = \ell(\fkv_0)\,, \quad \text{for all } \,\fkv_0\in V_0\,.
  \end{aligned}
 \right. &&
\end{flalign*}
\end{theorem}

In the theorem above, we interpret $U_0$ to be the space of field solution variables (from the original unbroken formulation) and $\hat{U}$ to be the space of interface variables.
Indeed, $b_0$ is the bilinear form from the original problem, while $\hat{b}$ is the contribution from the interface variables.
For instance, the primal formulation in \eqref{eq:PrimalBrokenFormulation},
\begin{equation*}
	b((u,\hat{\sigma}_\fkn),v) = (\C:\nabla u,\nabla v)_\mesh - \langle\hat{\sigma}_\fkn,\tr_{\grad}v\rangle_{\bdry\mesh}\,,
\end{equation*}
is decomposed into $b_0(u,v) + \hat{b}(\hat{\sigma}_\fkn,v)$ by
\begin{equation*}
 b_0(u,v) = (\C:\nabla u,\nabla v)_\mesh\,,\qquad
 \hat{b}(\hat{\sigma}_\fkn,v) = - \langle\hat{\sigma}_\fkn,\tr_{\grad}v\rangle_{\bdry\mesh}\,.
\end{equation*}
Next, observe that by Lemma \ref{lem:H1Subspace},
\begin{equation*}
  V_0 = \{ v\in \bfH^1_{\Gamma_0}(\mesh) \mid \langle\hat{\sigma}_\fkn,\tr_{\grad}v\rangle_{\bdry\mesh} = 0
  	\,\,\text{for all}\,\, \hat{\sigma}_\fkn\in \bfH^{-\frac{1}{2}}_{\Gamma_1}(\bdry\mesh) \}
  		= \bfH^1_{\Gamma_0}(\Omega)\,.
\end{equation*}
Moreover, we immediately satisfy \ref{item:Assumption2} with $\hat{\gamma}=1$ by use of identity \eqref{eq:InterfaceNorms},
\begin{equation*}
	\|\hat{\sigma}_\fkn\|_{\bfH^{-\frac{1}{2}}(\bdry\mesh)}
			=\sup_{v\in\bfH^1(\mesh)\setminus\{0\}}\widefrac[0pt]{|\langle\hat{\sigma}_\fkn,\tr_{\grad}v\rangle_{\bdry\mesh}|}
				{\qquad\|v\|_{\bfH^1(\mesh)}}\,.
\end{equation*}
Furthermore, whenever $\Gamma_0\neq\varnothing$, we satisfy \ref{item:Assumption1} by Corollary \ref{cor:allwellposed}, while
\begin{equation*}
  V_{00} = \{ v\in \bfH^1_{\Gamma_0}(\Omega) \mid (\C:\nabla u, \nabla v)_\mesh = 0
  	\,\,\text{for all}\,\,u \in \bfH^1_{\Gamma_0}(\Omega) \}
  		= \{ v\in \bfH^1_{\Gamma_0}(\Omega) \mid \nabla v = 0 \}
  			= \{0\}\,.
\end{equation*}
Hence, by Theorem \ref{thm:BrokenWellposedness}, we have guaranteed existence and uniqueness of a solution in the broken primal formulation, $(\scP_\mesh)$.

\begin{corollary}
\label{cor:allbrokenwellposed}
Let $\,\Gamma_0$ be relatively open in $\bdry\Omega$.
If $\,\Gamma_0\neq\varnothing$, then the broken variational formulations $(\scS_\mesh)$, $(\scU_\mesh)$, $(\scD_\mesh)$, $(\scM_\mesh)$ and $(\scP_\mesh)$ are well-posed.
\end{corollary}

\begin{proof}
Continue as with the primal formulation in each of the other cases.
In general, use Lemma~\ref{lem:CharacterizationOfTraces} to define $V_0$ in concrete terms, and make use of the identities in \eqref{eq:InterfaceNorms} to satisfy \ref{item:Assumption2}.
The satisfaction of \ref{item:Assumption1} follows from Corollary \ref{cor:allwellposed}, while a simple calculation yields that $\ell|_{V_{00}}=0$ in all cases.
Finally, using Theorem \ref{thm:BrokenWellposedness}, one concludes that there exists a unique solution for each broken formulation, meaning they are all well-posed.
\end{proof}










\section{Minimum residual methods}
\label{sec:MinResidual}

\subsection{Optimal stability} 
\label{sec:derivation}

We consider a general variational formulation.
Let $U$ and $V$ be Hilbert spaces over a fixed field $\mathbb{F}\in\{\R,\mathbb{C}\}$.
If $\mathbb{F}=\R$ (or $\mathbb{C}$), allow $b:U\times V \to \mathbb{F}$ to be a bilinear (or sesquilinear) form and $\ell\in V'$ a continuous linear form.
We are interested in the abstract variational problem
\begin{equation}
\label{eq:bilinearFormEQ}
 \left\{
  \begin{aligned}
   &\text{Find } \fku\in U\,,\\
   &b(\fku,\fkv) = \ell(\fkv)\,,\quad \text{for all }\, \fkv\in V\,,
  \end{aligned}
 \right.
\end{equation}
which we assume to be well-posed.
As demonstrated in the previous two sections, for a single linear elasticity problem, we have plenty of candidates for the forms $b$ and $\ell$.

If $\mathbb{F}=\R$ (or $\mathbb{C}$), observe that the bilinear (or sesquilinear) form, $b$, uniquely defines a continuous linear (or antilinear) operator $B:U\to V'$ such that $\langle B\fku,\fkv\rangle_{V'\times V} = b(\fku,\fkv)$.
Therefore \eqref{eq:bilinearFormEQ} may be reinterpreted as the operator equation
\begin{equation}
\label{eq:linearFormEQ}
 \left\{
  \begin{aligned}
   &\text{Find } \fku\in U\,,\\
   &B\fku = \ell\,.
  \end{aligned}
 \right.
\end{equation}

Let $U_h\subseteq U$ be the trial space we have chosen to represent our solution with. Then, out of all elements of $U_h$ it is desirable to find the best solution to the given problem.
To this end, we seek to find the solution to the following minimization problem on the residual,
\begin{equation}
\label{eq:MinResidual}
	\fku_h = \argmin_{\fku\in U_h} \|B\fku-\ell\|^2_{V'}\,.
\end{equation}
Solving this equation will give us a best approximation (dependent on the norm $\|\cdot\|_V$) to the solution of \eqref{eq:bilinearFormEQ} in our truncated energy space $U_h$.
Vanishing of the first variation at the minimizer implies that it satisfies the variational equation
\begin{equation}
\label{eq:FirstVarMinimization}
	\big(B\fku_h-\ell,B\delta \fku \big)_{V'} = 0\,, \quad \text{for all }\, \delta \fku \in U_h\,.
\end{equation}
At this point we are left with an inner product in a dual space which does not lend itself easily to computation.
However, we can transform this equation to that of an inner product over $V$ by recalling the Riesz map, $R_V : V\to V'$.
Indeed, the Riesz representation theorem guarantees the unique existence of such a linear isometric isomorphism, $R_V$, which satisfies the equation
\begin{equation}
\label{eq:Riesz}
 	\langle R_V \fkv , \delta \fkv\rangle_{V'\times V} = (\fkv,\delta \fkv)_V\,, \quad\text{for all }\, \fkv,\delta \fkv\in V\,.
\end{equation}
Using the identity above, we may rewrite \eqref{eq:FirstVarMinimization} as
\begin{equation}
\label{eq:NormalEquations}
	\big(R_V^{-1}(B\fku_h-\ell),R_V^{-1} B\delta \fku \big)_V = 0\,, \quad \text{for all }\, \delta \fku \in U_h\,.
\end{equation}
This is called the \textit{normal equation}, while \eqref{eq:FirstVarMinimization} is called the \textit{dual normal equation}.

Defining the \textit{error representation function}, $\psi = R_V^{-1}(B\fku_h-\ell)\in V$, the same problem is written as
\begin{equation}
\label{eq:SaddlePointNormalEq}
 \left\{
  \begin{alignedat}{3}
   \multicolumn{1}{l}{$\text{Find } \fku_h\in U_h\,,\, \psi\in V\,,$} \hspace{-300pt}& &&\\
   -( \psi, \fkv )_V + b(\fku_h, \fkv) &= \ell(\fkv) \,, \quad &&\text{for all }\, \fkv \in V\,,\\
   b(\delta \fku, \psi) &= 0\,, \quad &&\text{for all }\, \delta \fku \in U_h\,,
  \end{alignedat}
 \right.
\end{equation}
where the first equation corresponds to the definition of $\psi$ in the form $R_V\psi=B\fku_h-\ell$, which is precisely the residual, and the second equation is the normal equation, \eqref{eq:NormalEquations}, rewritten.

Alternatively, writing $\delta \fkv=R_V^{-1} B\delta \fku$ in \eqref{eq:NormalEquations} yields the problem
\begin{equation}
\label{eq:OptimalTestNormalEq}
 \left\{
  \begin{alignedat}{1}
   &\text{Find } \fku_h\in U_h\,,\\
   &b(\fku_h,\delta \fkv) = \ell(\delta \fkv)\,, \quad \text{for all }\, \delta \fkv \in V^\opt\,,
  \end{alignedat}
 \right.
\end{equation}
where we define the \emph{optimal test space}, $V^\opt = R_V^{-1}BU_h$, which obviously satisfies $\dim(V^\opt)=\dim(U_h)$.
With these discrete trial and test spaces, it can be shown the stability properties of the original problem are reproduced in the sense that
\begin{equation}
	\inf_{\fku\in U\setminus\{0\}}\sup_{\fkv\in V\setminus\{0\}}\frac{|b(\fku,\fkv)|}{\|\fku\|_U\|\fkv\|_V}
		=\inf_{\fku\in U_h\setminus\{0\}}\sup_{\fkv\in V^\opt\setminus\{0\}}\frac{|b(\fku,\fkv)|}{\|\fku\|_U\|\fkv\|_V}\,,
\end{equation}
and for this reason the discrete problem is said to have \textit{optimal stabililty}.


Clearly \eqref{eq:NormalEquations}, \eqref{eq:SaddlePointNormalEq} and \eqref{eq:OptimalTestNormalEq} are equivalent.
They differ only in their interpretation.




\subsection{The DPG methodology} 
\label{sec:the_dpg_method}

In general, neither \eqref{eq:NormalEquations}, \eqref{eq:SaddlePointNormalEq} nor \eqref{eq:OptimalTestNormalEq} are amenable to computation, because in practice we cannot test with infinite $\fkv\in V$ to invert the Riesz map exactly.
For this reason we must seek an approximate solution to them by considering only a truncated, yet large, \textit{enriched test space} $V^\enr\subseteq V$ satisfying $\dim(V^\enr)>\dim(U_h)$.

For example, \eqref{eq:SaddlePointNormalEq} becomes
\begin{equation}
\label{eq:SaddlePointDPG}
 \left\{
  \begin{alignedat}{3}
   \multicolumn{1}{l}{$\text{Find } \fku_h\in U_h\,,\, \psi_h\in V^\enr\,,$} \hspace{-400pt}\\
   -( \psi_h, \fkv )_V + b(\fku_h, \fkv) &= \ell(\fkv) \,, \quad &&\text{for all }\, \fkv \in V^\enr\,,\\
   b(\delta \fku, \psi_h) &= 0\,, \quad &&\text{for all }\, \delta \fku \in U_h\,.
  \end{alignedat}
 \right.
\end{equation}
This can be rewritten as a linear system
\begin{equation}
\label{eq:SaddlePointSystem}
	\begin{pmatrix}-R_{V^\enr} & B\\B^\T & 0\end{pmatrix}
  \begin{pmatrix}\psi_h\\\fku_h\end{pmatrix}
  =\begin{pmatrix}
  	\ell\\
		0
  \end{pmatrix}\,,  
\end{equation}
where, naturally,
\begin{equation}
\label{eq:DiscreteRiesz}
 	\langle R_{V^\enr} \fkv , \delta \fkv\rangle_{V'\times V} = (\fkv,\delta \fkv)_V\,, \quad \text{for all }\, 
 		\fkv,\delta \fkv\in V^\enr\,,
\end{equation}
and $B^\T=B':V\to U'$ is the transpose of $B$ defined by $\langle B^\T \fkv,\fku\rangle_{U'\times U}=b(\fku,\fkv)=\langle B\fku,\fkv\rangle_{V'\times V}$.
Static condensation of \eqref{eq:SaddlePointSystem} to remove $\psi_h$ leads to the \textit{discrete normal equations},
\begin{equation}
\label{eq:DiscreteNormalEq}
	B^\T R_{V^\enr}^{-1} B \fku_h = B^\T R_{V^\enr}^{-1} \ell\,.
\end{equation}
One can similarly attain these equations using \eqref{eq:OptimalTestNormalEq} and testing with the approximate optimal test space $V^\opt_h = R_{V^\enr}^{-1}BU_h$, or equivalently using \eqref{eq:NormalEquations} by first considering the residual minimization problem $\fku_h=\argmin_{\fku\in U_h} \| B\fku-\ell \|^2_{(V^\enr)'}$.

At this point, one may now observe that solving the discrete normal equations using traditional unbroken test spaces would be ineffective because it involves inverting a large linear system resulting from $R_{V^\enr}$. 
To circumvent this issue, we reformulate the problem with broken test spaces and so allow the inversion calculation of $R_{V^\enr}$ to be localized.
This gives more efficient (and possibly parallel) computations at the cost of adding more degrees of freedom through extra interface unknowns.
Indeed, with broken test spaces, the normal equations need only be computed element-wise since the matrix representing the Riesz map has an easily invertible diagonal block structure with each block representing an element.
The use of broken test spaces coupled with the discrete normal equations aiming to approximate optimal stability properties constitutes the \textit{DPG methodology}.

\begin{remark}
The modus operandi in DPG computations has been to construct trial spaces with polynomial orders inferred from an order $p$ discrete exact sequence of the first type approximating
\begin{equation*}
	H^1 \xrightarrow{\,\,\nabla\,\,} H(\curl) \xrightarrow{\nabla\times} H(\div) \xrightarrow{\,\nabla\cdot\,} L^2 \, .
\end{equation*}
Essentially, when we take order $p$ polynomials for $H^1$ shape functions, we use order $p-1$ polynomials for our $L^2$ shape functions.
For the $H(\div)$ shape functions, we use order $p$ polynomials whose normal trace is of order $p-1$ on each face. 
For the traces, we inspect the trace operators and choose to use $p$ order polynomials for the $H^{\frac{1}{2}}$ variables and $p-1$ order polynomials for the $H^{-\frac{1}{2}}$ variables.
%
%
%
To construct the enriched test space $V^\enr$, it has become customary to choose a uniform $p$ enrichment over the order taken by the trial space variables.
Denoting the enrichment order $\mathrm{d}p$, we choose $p+\mathrm{d}p$ order test functions.
Naturally, this explanation is an understatement of a more complex subject, since the appropriate exact sequence spaces of polynomials differ considerably depending on the element shape \cite{hpbook2,Fuentes2015}.
\end{remark}

\begin{remark}
\label{rmk:Fortinapproximation}
A natural question is whether the approximate optimal test space, $V^\opt_h$ is an accurate representation of $V^\opt$ for a given enrichment, $\dd p$.
Or, similarly, whether the solution of \eqref{eq:DiscreteNormalEq} is sufficiently close to the solution of \eqref{eq:NormalEquations}. A means of analysis of this question has been presented in the context of Fortin operators in \cite{gopalakrishnan2014analysis,Carstensen15,Nagaraj2015} and so far, several different problems have been studied.
In fact, through this analysis for linear elasticity in the ultraweak setting with isotropic materials, a very practical $\mathrm{d}p \leq 3$ has been shown to be sufficient for optimal convergence rates in 3D computations \cite{gopalakrishnan2014analysis}.
For us to include a similar account for each variational formulation we consider here would be substantially distracting.
As it will be observed in Section \ref{sec:Numerics}, in our work, just $\mathrm{d}p = 1$ was sufficient to obtain the desired convergence rates in all of our computations.
\end{remark}

\begin{remark}
Solving the saddle point problem, \eqref{eq:SaddlePointDPG}, outright, also has some benefits.
Indeed this system, albeit larger than the discrete normal equations, can be solved with standard finite elements due to its symmetric functional setting with $U_h\times V^\enr$ used for both trial and test spaces.
This approach has been explored in \cite{dahmen2012adaptive,cohen2012adaptivity}.
\end{remark}


\subsection{\texorpdfstring{$L^2$}{L2} test spaces} 
\label{sec:hybrid_dpg}

In this section we abandon any distinction between $\bfL^2(\Omega)$, $\bfL^2(\Omega;\Sym)$ and $\bfL^2(\Omega;\Skw)$, and liberally refer to any product of them as simply $L^2$.
When some of the test variables are in $L^2$ it is possible to exploit that $(L^2)'\cong L^2$ to avoid, at least to some degree, the discrete inversion of the Riesz map.

The most salient case occurs with the strong formulation, \eqref{eq:TrivialBrokenFormulation}, where $V = L^2$.
Here, the first variation, \eqref{eq:FirstVarMinimization}, which is equivalent to \eqref{eq:NormalEquations}, \eqref{eq:SaddlePointNormalEq} and \eqref{eq:OptimalTestNormalEq}, is written as
\begin{equation}
\label{eq:LeastSquares}
 \left\{
  \begin{aligned}
   &\text{Find } \fku_h\in U_h\,,\\
   &(B\fku_h,B\delta \fku)_{(L^2)'} = (\ell,B\delta \fku)_{(L^2)'}\,, \quad \text{for all }\, \delta \fku \in U_h\,.
  \end{aligned}
 \right.
\end{equation}
In this case the $(L^2)'$ inner product \textit{is} amenable to computation and is simply the usual $L^2$ inner product after trivially identifying $B\fku_h$, $B\delta \fku$ and $\ell$ with $L^2$ functions.
Note this immediate identification is precisely the inverse Riesz map, which is otherwise usually nontrivial.
This simplified method already exists in the literature and is known as the first order system least squares formulation (FOSLS).
Observe that in this case (assuming exact integration) the optimal stability of the original formulation is reproduced exactly due to the exact inversion of the Riesz map, thus avoiding the numerical error that arises when discretizing with an enriched test space, $V^\enr$.



When part of the test space is in $L^2$, such as in \eqref{eq:Mixed1BrokenFormulation} and \eqref{eq:Mixed2BrokenFormulation}, similar optimizations are also possible, but only in the $L^2$ part of the test space, where the Riesz map is trivial.
This both lowers computational cost and helps to better approach optimal stability.
We now present a derivation for those cases in a general setting.

Let $W$ be a Hilbert space and assume the test space has the form $V = W \times L^2$ with the Hilbert norm $\|(\fkv_W,\fkv_{L^2})\|_V^2 = \|\fkv_W\|_W^2+\|\fkv_{L^2}\|_{L^2}^2$.
We thereby decompose $B = B_W \times B_{L^2}$ and $\ell = \ell_W \times \ell_{L^2}$ and rewrite the normal equation, \eqref{eq:NormalEquations}, in a decoupled form
\begin{equation}
\label{eq:NormalEquationHybrid}
	\big(R_W^{-1}(B_W \fku_h-\ell_W),R_W^{-1} B_W\delta \fku \big)_W+\big(B_{L^2}\fku_h-\ell_{L^2},B_{L^2}\delta \fku\big)_{(L^2)'}=0\,,
 	\quad \text{for all } \, \delta \fku \in U_h\,.
\end{equation}
After defining $\psi_W = R_W^{-1}(B_W \fku_h-\ell_W)$ and taking $\langle \cdot,\cdot \rangle = \langle \cdot,\cdot \rangle_{W'\times W}$, the duality pairing between $W$ and $W'$, \eqref{eq:NormalEquationHybrid} leads to the system
\begin{equation}
\label{eq:PseudoSaddlePointHybrid}
  \left\{
  \begin{alignedat}{3}
   \multicolumn{1}{l}{$\text{Find } \fku_h\in U_h\,,\, \psi_W\in W\,,$}\\
   -( \psi_W, \fkw )_W + \langle B_W \fku_h, \fkw\rangle
    &= \langle \ell_W, \fkw\rangle
     \,, \quad &&\text{for all }\, \fkw \in W\,,\\
   \langle B_W\delta \fku, \psi_W\rangle
    + \big(B_{L^2} \fku_h,B_{L^2}\delta \fku \big)_{(L^2)'} &= \big(\ell_{L^2},B_{L^2}\delta \fku \big)_{(L^2)'}\,,
    	\quad &&\text{for all }\, \delta \fku \in U_h\,.
  \end{alignedat}
 \right.
\end{equation}
Alternatively, identifying $\delta \fkw=R_W^{-1} B_W\delta \fku$ in \eqref{eq:NormalEquationHybrid} yields
\begin{equation}
\label{eq:OptimalTestNormalEqHybrid}
 \left\{
  \begin{alignedat}{3}
   &\text{Find } \fku_h\in U_h\,,\\
   &\langle B_W\fku_h,\delta\fkw\rangle\!+\!(B_{L^2}\fku_h,B_{L^2}\delta \fku)_{(L^2)'}\!=\!\langle\ell_W,\delta\fkw\rangle
    \!+\! (\ell_{L^2},B_{L^2}\delta\fku)_{(L^2)'}\,,\quad\!\!\text{for all }\,(\delta \fkw,\delta\fku) \!\in\! W^\opt\,\!,\!\!
  \end{alignedat}
 \right.
\end{equation}
where the optimal graph test space is $W^\opt=\{(R_W^{-1}B_W\delta\fku,\delta\fku)\mid\delta\fku\in U_h\}\subseteq W\times U_h$.


\begin{remark}
In some other limited scenarios, the optimal test space can be exactly computed a priori \cite{demkowicz2011analysis}.
In particular, a symmetric functional setting and a coercive bilinear form which is realized as an inner product on the (trial or test) space minimizing the residual, results in the classic Bubnov-Galerkin method. 
\end{remark}


\subsection{Adaptivity} 
\label{sec:adaptivity}

One big advantage of minimum residual methods is that they have a built-in a posteriori error estimator which can be used in an adaptive mesh refinement algorithm.
This is because, implicitly, by minimizing the residual we are also minimizing the error in a problem-dependent energy norm,
\begin{equation}
\label{eq:EnergyNorm}
 	\|\cdot\|_E = \|B(\cdot)\|_{V'}=\sup_{\fkv\in V\setminus\{0\}} \frac{b(\cdot,\fkv)}{\|\fkv\|_V}\,.
\end{equation}
If $\fku$ is the exact solution to the abstract variational problem, \eqref{eq:bilinearFormEQ}, satisfying $B\fku=\ell$, and $\fku_h$ is defined by the minimization of the residual in $U_h$ as in \eqref{eq:MinResidual}, then
\begin{equation}
\label{eq:EnergyNormResidual}
 	\|\fku-\fku_h\|_E = \|B\fku-B\fku_h\|_{V'}=\|B\fku_h-\ell\|_{V'},
\end{equation}
which is precisely the minimum residual attained by $\fku_h$ in \eqref{eq:MinResidual}.
Hence, if the trial space is refined such that it is larger and embedded in the previous trial space, then the minimum residual in \eqref{eq:MinResidual} (and thus $\|\fku-\fku_h\|_E$) will be smaller provided that $V$ does not change with every refinement, and that the norm in $V'$ is computed exactly.
Under these assumptions, the residual will \textit{always} decrease with each successive refinement.
Nevertheless, when using mesh-dependent broken test spaces and interface trial variables, as the mesh is refined, the embeddings of trial spaces may not hold, and the test norms will change.
In theory, this may cause the residual to increase \cite{Heuer15}.
In practice, however, the residual is usually observed to decrease with each mesh refinement.

Due to its (typically) decreasing behavior, the residual is an ideal a posteriori error estimator as long as it can be expressed as a sum of residual contributions from each element.
This allows to detect, via some discrete criterion, which individual elements need to be further refined (see any simple greedy algorithm).
Unfortunately, in general, we cannot exactly compute
\begin{equation}
\label{eq:ErrorComputation}
  \|B\fku_h-\ell\|_{V'}^2=\langle B\fku_h-\ell,R_V^{-1}(B\fku_h-\ell)\rangle_{V'\times V}\,,
\end{equation}
due to the nature of the inverse Riesz map.
In a discrete setting with a truncation $V^\res\subseteq V$, the \textit{approximate} residual becomes $\|B\fku_h-\ell\|_{(V^\res)'}^2=\langle B\fku_h-\ell,R_{V^\res}^{-1}(B\fku_h-\ell)\rangle_{V'\times V}=(B\fku_h-\ell)^\T R_{V^\res}^{-1}(B\fku_h-\ell)$, where $B\fku_h-\ell$ represents a vector with each component being the duality pairing with a basis element of $V^\res$.
In general, the approximate residual cannot be expressed as a sum of element contributions.
However, the use of broken test spaces allows for a completely localized and parallelizable computation of the element-wise residual contributions.
Therefore, the DPG methodology is particularly convenient to implement residual-based adaptive refinement strategies.

%


\begin{remark}
Although we intuitively expect the \textit{exact} residual will decrease with each successive refinement, this may not happen with the \textit{approximate} residual which is actually computed.
The quality of the approximation is dependent upon the orders used for the enriched test space, $V^\enr$, and the truncated residual test space, $V^\res\subseteq V$.
We remark that the truncation used for the residual computation, $V^\res\subseteq V$, can be different from the enriched test space used to solve the discrete problem, $V^\enr\subseteq V$.
Indeed, choosing a larger and fixed truncation for $V^\res$ can improve the accuracy of the a posteriori error estimator, while it also facilitates comparison of the residual computation when solving with different polynomial orders.
Again, the cost of making this choice is not greatly affected provided all residual computations are done in parallel.
\end{remark}

\begin{remark}
In the cases where part of the test space is $L^2$, the residual computation can be further simplified by computing the norm of the residual straight from the inner product.
As in Section \ref{sec:hybrid_dpg}, consider a test space $V=W \times L^2$ and decompositions $B=B_W\times B_{L^2}$ and $\ell=\ell_W \times\ell_{L^2}$.
Then the residual is
\begin{equation}
	\|\fku-\fku_h\|_E^2 = \langle B_W\fku_h-\ell_W,R_W^{-1}(B_W\fku_h-\ell_W)\rangle_{W'\times W}+
		(B_{L^2}\fku_h-\ell_{L^2},B_{L^2}\fku_h-\ell_{L^2})_{(L^2)'}\,,
\end{equation}
where the inner products in $(L^2)'$ can be computed exactly by identifying $B_{L^2}\fku_h-\ell_{L^2}$ with elements of $L^2$, while the term involving $W$ is approximated with a discrete truncation $W^\res\subseteq W$ as described before.
%
%
\end{remark}

\section{Numerical experiments}
\label{sec:Numerics}

Both $\bfH^1$ and $\bfH(\div)$ (as opposed to $\bfH(\div;\Sym)$) are essentially three copies of $H^1$ and $H(\div)$ respectively, and similar assertions apply to $\bfL^2$, $\bfL^2(\Skw)$ and $\bfL^2(\Sym)$ which themselves are a number of copies of $L^2$.
Seeing these as such, in our computations those spaces were discretized using the arbitrary order conforming shape functions defined in \cite{Fuentes2015}.
As previously mentioned, the polynomial orders are naturally determined by a discrete exact sequence of order $p$ associated to a particular element type. 
Without dwelling into the details, $\bfH^1$ and $\bfH(\div)$ were discretized by specific order $p$ polynomials while $\bfL^2$, $\bfL^2(\Skw)$ and $\bfL^2(\Sym)$ were discretized using specific order $p-1$ polynomials (even though the order of the \textit{sequence} is $p$).
Meanwhile, the trace variables in $\bfH^{\frac{1}{2}}$ and $\bfH^{-\frac{1}{2}}$ were discretized simply by isolating the trace of the $\bfH^1$ and $\bfH(\div)$ shape functions which are nonzero at some part of the boundary and so ultimately they were polynomials of order $p$ and $p-1$ respectively (see \cite{hpbook2,Fuentes2015} for more details).
Chosen in this way, the discrete spaces satisfy polynomial interpolation inequalities and thereby assuming discrete stability of the numerical methods, the typical convergence rates are ensured.

The trial spaces were always discretized from a sequence of order $p$, while the enriched test spaces, $V^\enr$, were always discretized from a sequence of order $p+\mathrm{d}p$. 
Here it is notable that $\mathrm{d}p=1$ was sufficient for all of our computations, irrespective of the variational formulation, and so this is the value used for all of the results given.
Whenever necessary, the residual is calculated with a truncated residual test space, $V^\res$, of fixed order $p+\mathrm{d}p=4$ to facilitate comparison between differing values of $p$ within a fixed variational formulation and fixed mesh.

In what follows, the DPG methodology was used as described in Section \ref{sec:the_dpg_method} to solve the strong formulation in \eqref{eq:TrivialBrokenFormulation}, the primal formulation in \eqref{eq:PrimalBrokenFormulation}, the mixed formulation in \eqref{eq:Mixed2BrokenFormulation} and the ultraweak formulation in \eqref{eq:UWeakBrokenFormulation}.
The formulation in \eqref{eq:Mixed1BrokenFormulation} is not computed with, because the primal formulation essentially fulfills its role.
With the objective of better approximating the optimal test spaces, both the strong and mixed formulations are implemented by inverting at least a part of the Riesz map exactly as detailed in Section \ref{sec:hybrid_dpg}.

The residual is $\|B\fku_h-\ell\|_{V'}$, where $B$ is the operator from a given variational formulation, $\ell$ is the linear functional of the formulation (basically the force $f$), $\fku_h$ is the computed trial variable (the tuple of unknowns including the trace variables) and $\|\cdot\|_{V'}$ is the norm in the dual test space $V'$.
The test space is always a Cartesian product of broken spaces among $\bfH^1_{\Gamma_0}(\mesh)$, $\bfH_{\Gamma_1}(\div,\mesh)$, $\bfL^2(\mesh)$, $\bfL^2(\mesh;\Sym)$ and $\bfL^2(\mesh;\Skw)$.
In our computations, each of these spaces was suited with its standard norm as defined in \eqref{eq:BrokenNorms}, so that $V$ was assumed to inherit the associated Hilbert norm.
Choosing different norms for the broken test space is possible, but this was not thoroughly analyzed in this work.
Meanwhile, the relative displacement error is $\frac{\|u-u_h\|}{\|u\|}$, where $u$ is the exact displacement and $u_h$ is the computed displacement, and where the norm $\|\cdot\|$ depends on the variational formulation being used ($\|\cdot\|_{H^1}$ with primal and strong formulations and $\|\cdot\|_{L^2}$ with ultraweak and mixed formulations).

The convergence results are always presented in terms the degrees of freedom $N_{\mathrm{dof}}$, instead of the symbolic size of the element $h=\mathcal{O}(\sqrt[3]{N_{\mathrm{dof}}})$, because this is a more reasonable metric when using adaptive meshes.
Therefore, all expected convergence rates in terms of $h$ should be divided by $3$ (since computations are done in 3D) to get the appropriate rates in terms of $N_{\mathrm{dof}}$ and viceversa.

\begin{remark}
In the context of the mixed formulations with weakly imposed symmetry, $(\scM)$ and $(\scM_\mesh)$, expressed in \eqref{eq:Mixed2Formulation} and \eqref{eq:Mixed2BrokenFormulation} respectively, some authors choose to nontrivially extend the compliance tensor, $\A$, from $\Sym$ to $\Mat$ \cite{mixedelas3d,Bramwell12}.
They do this to ensure that $(\A:\sigma,\sigma)_\Omega$ remains positive definite on $\bfL^2(\Omega,\Mat)$ (and not only on $\bfL^2(\Omega,\Sym)$).
However, in this work, we chose to extend the compliance tensor trivally, so that $\A|_\Skw=0$ (see Section \ref{sec:variationalequations}).
This did not pose any limitations in the infinite-dimensional setting while proving the well-posedness of the mixed variational formulations, $(\scM)$ and $(\scM_\mesh)$ (see Corollaries \ref{cor:allwellposed} and \ref{cor:allbrokenwellposed} and Appendix \ref{app:WellPosedness}).
For the practical DPG methodology, where the test space is designed to approximate the optimal test space, as pointed out in Remark \ref{rmk:Fortinapproximation}, one can show that for a large enough enrichment (i.e. value of $\mathrm{d}p$) the problem remains well-posed.
Thus, it is valid to extend $\A$ trivially, as this does not affect the presence of discrete stability.
\end{remark}

\subsection{Smooth solution}

To test all variational formulations we first tackled a problem with a smooth solution.
We considered the cubic domain $\Omega=(0,1)^3$ and the displacement manufactured solution
\begin{equation}
	u_i(x_1,x_2,x_3)=\sin(\pi x_1)\sin(\pi x_2)\sin(\pi x_3)\,,
	\label{eq:exactdisplacementcube}
\end{equation}
where $i=1,2,3$.
For the stiffness and compliance tensors, $\C$ and $\A$, we considered a simple isotropic material with the nondimensionalized Lam\'e parameters $\lambda=\mu=1$. This induces manufactured stress $\sigma$ and force $f$ defined by the constitutive relation and momentum conservation equations in \eqref{eq:LinElastSystem}.
Loaded with this manufactured body force, we considered each variational formulation with the pure displacement boundary conditions $u_0=0$ (taken from the exact solution).

The convergence of each method is determined by solving successively on a series of uniformly refined meshes where the initial mesh is composed of five tetrahedra.
The convergence can be analyzed in terms of the displacement error or of the residual.
Since the exact solution is smooth, the expected rate of convergence with respect to $h$ is precisely $p$, where $p$ is the order of the discrete sequence associated to the trial space.
In terms of $N_{\mathrm{dof}}$, the expected rate is $\frac{p}{3}$.

\begin{figure}[!ht]
\begin{center}
\includegraphics[trim=3.4cm 0cm 1cm 0cm,clip=true,scale=0.39]{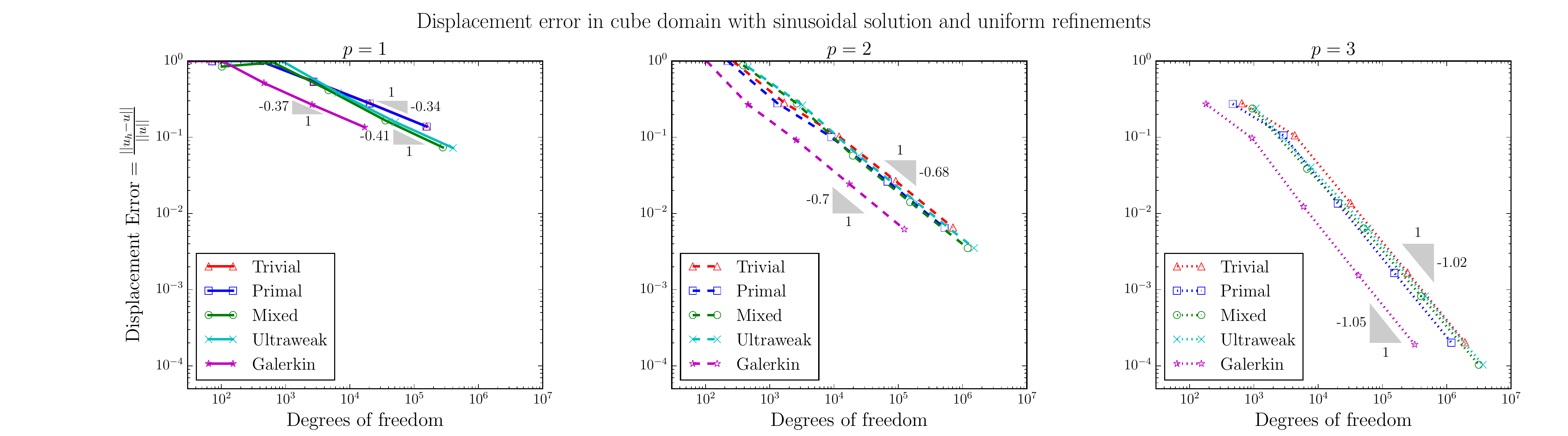}
\caption{Relative displacement error as a function of the degrees of freedom after uniform tetrahedral refinements in the cube domain with a smooth solution.}
\label{fig:cube-uniform-error}
\end{center}
\end{figure}

The results for the relative error are presented in Figure \ref{fig:cube-uniform-error} for $p=1,2,3$ and they are shown for the four previously mentioned DPG formulations alongside the classical Galerkin method given by \eqref{eq:PrimalFormulation}.
In this case, the convergence rates are precisely as expected for all methods.
All DPG methods seem to behave very similarly, while the Galerkin method stands out for using less degrees of freedom (since it involves no trace variables).
The results in terms of the residual show a similar behavior and are illustrated in Figure \ref{fig:cube-uniform}.
Note there are no results for the residual of the classical Galerkin method because we have not implemented a way of calculating it without using broken test spaces.

\begin{figure}[!ht]
\begin{center}
\includegraphics[trim=3.4cm 0cm 1cm 0cm,clip=true,scale=0.39]{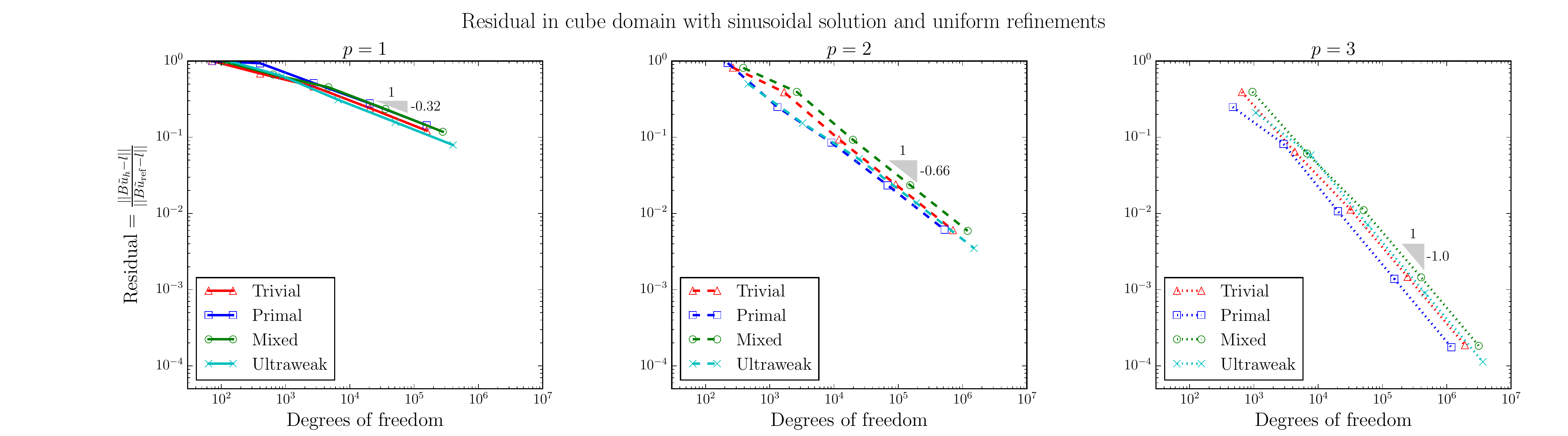}
\caption{Residual as a function of the degrees of freedom after uniform tetrahedral refinements in the cube domain with a smooth solution.}
\label{fig:cube-uniform}
\end{center}
\end{figure}

\subsection{Singular solution}

Perhaps a more interesting test is that of a problem with a singular solution.
A typical domain to ellicit these solutions is the L-shape domain.
A careful presentation in \cite[\S2.21--26]{CokerFilon} considers a 3D domain under plane strain or \textit{averaged} plane stress conditions, where in both cases the analysis effectively reduces it to a two dimensional problem.
Indeed, the L-shape domain example is prevalent as a 2D singular problem in the literature \cite{Vasilopoulos,Bramwell12,Grisvard}, especially the averaged plane stress case, which is elaborate to reformulate back into 3D \cite[\S2.26]{CokerFilon}.
For this reason, in this work we consider the plane \textit{strain} case in 3D.

\begin{figure}[!ht]
\begin{center}
\includegraphics[scale=0.5]{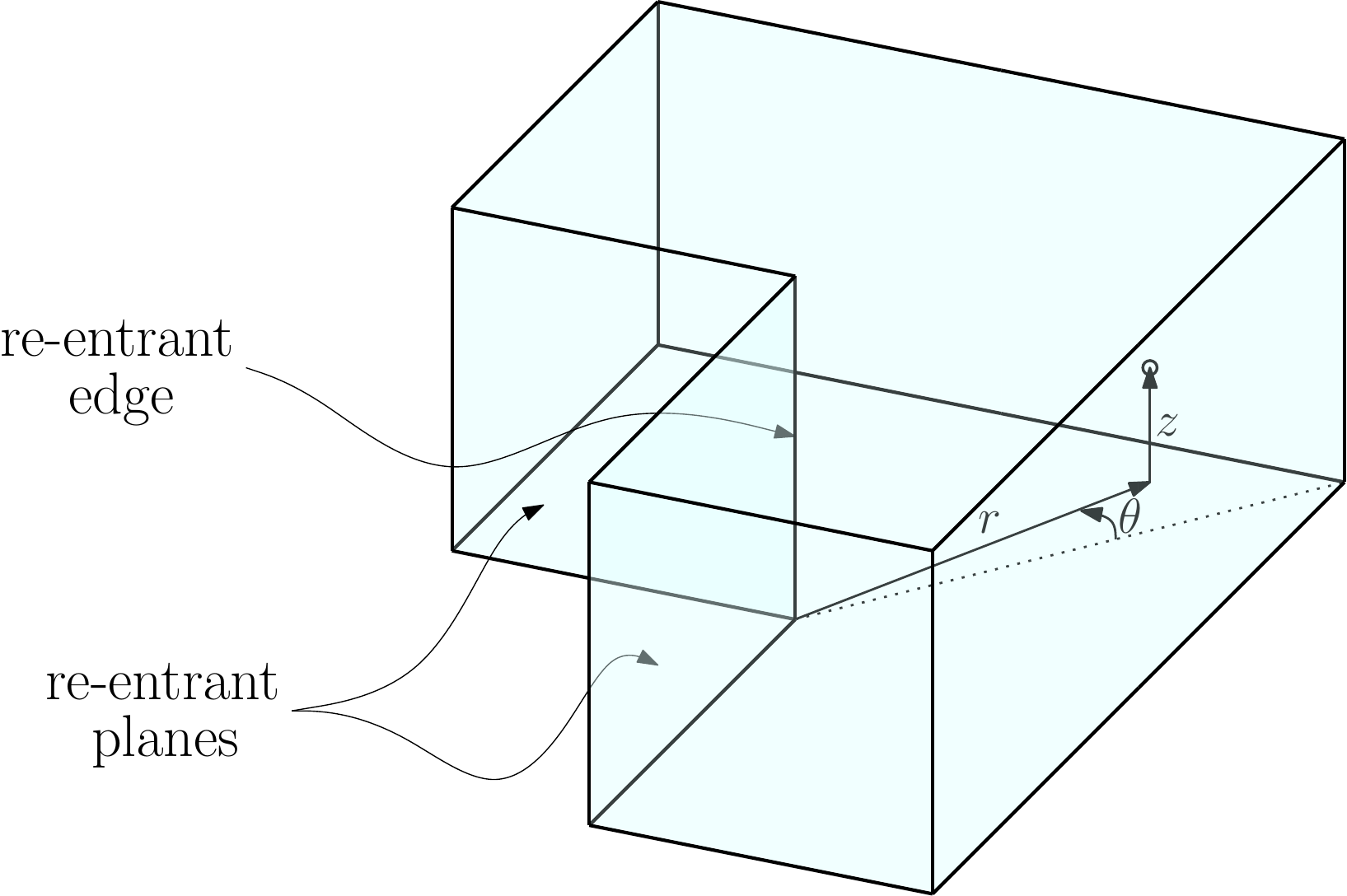}
\caption{L-shape domain in a cylindrical system of coordinates.}
\label{fig:Lshape-domain}
\end{center}
\end{figure}

As depicted in Figure \ref{fig:Lshape-domain}, we considered an L-shape domain composed of three unit cubes and a cylindrical system of coordinates, $(r,\theta,z)$, such that the re-entrant edge passes through the origin and aligns with the $z$-axis, while the re-entrant planes align with $\theta=\pm\frac{3}{4}\pi$.

Using Airy functions (see \cite{Vasilopoulos}) one can obtain general expressions for the displacement components in polar coordinates of a homogeneous isotropic elastic body in equilibrium, so that $-\div(\C:\varepsilon(u))=f=0$. 
These are
\begin{equation}
	\begin{aligned}
		u_r(r,\theta)&=\frac{1}{2\mu}r^a\Big(-(a+1)F(\theta)+(1-\nu)G'(\theta)\Big)\,,\\
		u_\theta(r,\theta)&=\frac{1}{2\mu}r^a\Big(-F'(\theta)+(1-\nu)(a-1)G(\theta)\Big)\,,\\
		u_z(r,\theta)&=0\,,
	\end{aligned}
	\label{eq:singularsoldisplacement}
\end{equation}
where $\nu=\frac{\lambda}{2(\lambda+\mu)}$ is the Poisson's ratio, $a$ is a constant, and
\begin{equation}
	\begin{aligned}
		F(\theta)&=C_1\sin((a+1)\theta)+C_2\cos((a+1)\theta)+C_3\sin((a-1)\theta)+C_4\cos((a-1)\theta)\,,\\
		G(\theta)&=-\frac{4}{a-1}\Big(C_3\cos((a-1)\theta)-C_4\sin((a-1)\theta)\Big)\,.
	\end{aligned}
\end{equation}
The nonzero stresses in polar coordinates satisfying the constitutive relation (and $\div\sigma=0$) are
\begin{equation}
	\begin{aligned}
		\sigma_{rr}(r,\theta)&=r^{a-1}\Big(F''(\theta)+(a+1)F(\theta)\Big)\,,\\
		\sigma_{\theta\theta}(r,\theta)&=a(a+1)r^{a-1}F(\theta)\,,\\
		\sigma_{r\theta}(r,\theta)&=-ar^{a-1}F'(\theta)\,,\\
		\sigma_{zz}(r,\theta)&=\lambda\tr(\epsilon(u))\,.
	\end{aligned}
	\label{eq:singularsolstress}
\end{equation}

Next, consider zero displacement boundary conditions at the re-entrant planes meaning that we want $u_r(r,\pm\frac{3}{4}\pi)=u_\theta(r,\pm\frac{3}{4}\pi)=0$.
The values of $C_1$, $C_2$, $C_3$, $C_4$ and $a$ are essentially \textit{chosen} to satisfy these boundary conditions.
Indeed, choosing $C_2=C_4=0$, $C_3=1$ and
\begin{equation}
	C_1=\frac{\Big(4(1-\nu)-(a+1)\Big)\sin\Big((a-1)\frac{3}{4}\pi\Big)}{(a+1)\sin\Big((a+1)\frac{3}{4}\pi\Big)}
	\label{eq:singularsolC1}
\end{equation}
guarantees that $u_r(r,\pm\frac{3}{4}\pi)=0$ regardless of the value of $a$.
After making this choice, the condition $u_\theta(r,\pm\frac{3}{4}\pi)=0$ becomes
\begin{equation}
	C_1(a+1)\cos\Big((a+1)\textstyle{\frac{3}{4}}\pi\Big)+\Big(4(1-\nu)+(a-1)\Big)\cos\Big((a-1)\textstyle{\frac{3}{4}}\pi\Big)=0\,.
	\label{eq:EquationForA}
\end{equation}
Moreover, since $\sigma$ has a common factor of $r^{a-1}$ it follows that $a>0$ is required to have $\sigma\in\bfL^2(\Sym)$, which
in turn implies $\sigma\in\bfH(\div;\Sym)$ in view of the intrinsic expression $\div\sigma=0$.
Furthermore, to have an actual singularity in the strains and stresses it is necessary for $a<1$.
Hence, $a$ is chosen to satisfy \eqref{eq:EquationForA}, with $a\in(0,1)$.

For steel, the Lam\'e parameters are $\lambda=123\,\mathrm{GPa}$ and $\mu=79.3\,\mathrm{GPa}$.
They yield $\nu\approx0.304$ and a constant $a\approx0.5946\in(0,1)$.
These values are used in our computations.
Regarding the boundary conditions, we impose displacement boundary conditions at the re-entrant planes, and stress (traction) boundary conditions at the other faces parallel to the $z$-axis.
The remaining two faces perpendicular to the $z$-axis are equipped with mixed boundary conditions where the displacement is restricted in the normal direction ($u_z=0$) and where the tangential components of the traction vanish.

\begin{remark}
Under averaged plane stress conditions the problem is extremely similar to the plane strain case.
The major difference is that the 2D displacements and stresses, $u_r$, $u_\theta$, $\sigma_{rr}$, $\sigma_{r\theta}$ and $\sigma_{\theta\theta}$, are actually \textit{averaged} quantities over the $z$ direction.
To solve the 2D problem for the averages simply consider the same equations as the plane strain case, but ignore $u_z$ and $\sigma_{zz}$, and change $\nu$ to $\frac{\nu}{1+\nu}$ in \eqref{eq:singularsoldisplacement}, \eqref{eq:singularsolstress}, \eqref{eq:singularsolC1} and \eqref{eq:EquationForA} (see \cite{Vasilopoulos}).
Recovering a 3D solution from the averaged quantities involves several calculations and is described in \cite[\S2.26]{CokerFilon}.
\end{remark}

\subsubsection{Uniform refinements}

The common factor of the stresses, $r^{a-1}$, actually implies that $\sigma$ is in a space of fractional order $s$, which roughly speaking corresponds to $s=1+(a-1)-\delta=a-\delta$, where $\delta>0$.
Since $a\in(0,1)$, it follows that under \textit{uniform} refinements the expected convergence rate with respect to $h$ is approximately $a$, meaning the expected convergence rate with respect to degrees of freedom $N_{\mathrm{dof}}$ is $\frac{a}{3}\approx0.1982$, regardless of the value of $p$.

\begin{figure}[!ht]
\begin{center}
\includegraphics[trim=3.4cm 0cm 1cm 0cm,clip=true,scale=0.39]{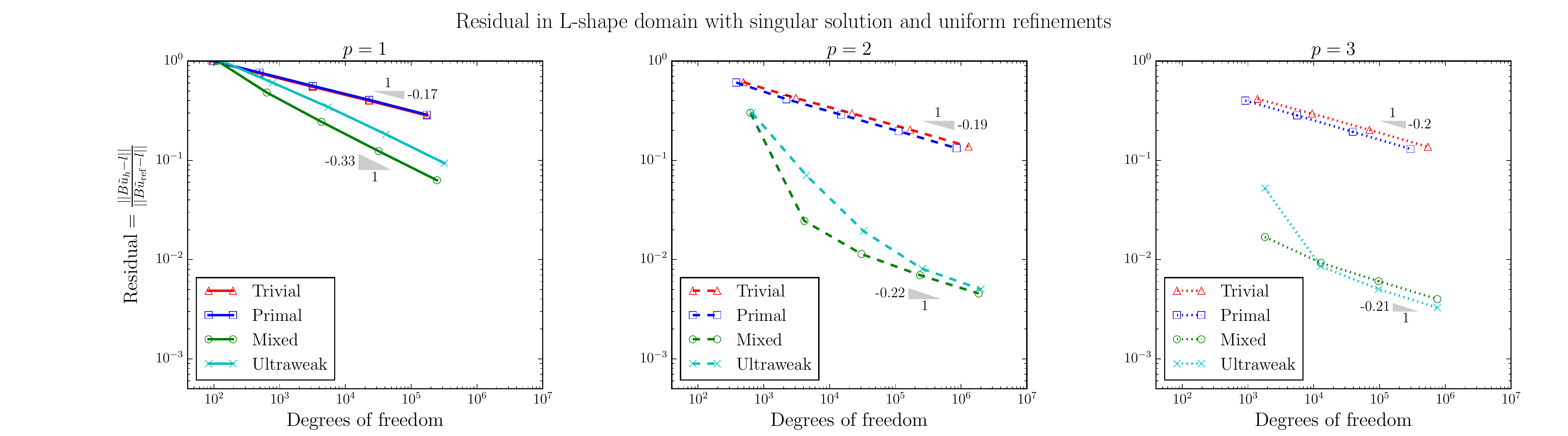}
\caption{Residual as a function of the degrees of freedom after uniform hexahedral refinements in the L-shape domain with a singular solution.}
\label{fig:Lshape-uniform}
\end{center}
\end{figure}

The uniform refinement results for the four variational formulations are presented in Figure \ref{fig:Lshape-uniform}.
As expected, the rates are very close to $\frac{a}{3}\approx0.1982$ when $p=2$ and $p=3$.
When $p=1$, the mixed and ultraweak methods seem to be converging at a higher rate (about $0.33$), but this is probably because it has not reached the asymptotic regime where it stabilizes to the expected rate.
For each formulation, as expected from the theory of minimum residual methods, the residual always goes down both when the mesh is refined for a fixed $p$ and also when $p$ is refined for a fixed mesh.
For example, the latter case is observed by looking at how the first point in the strong formulation (corresponding to the fixed initial five-element mesh) decreases in value as the order grows from $p=1$ (left plot) to $p=3$ (right plot).
This comparison is valid in the discrete setting only because a fixed value of $p+\mathrm{d}p=4$ was used to compute the residual in all cases.

\subsubsection{Adaptive refinements}

To prevent the proliferation of degrees of freedom and to have some form of theoretical background we use anisotropic refinements such that no refinements are done in the $z$ direction, where $u_z=0$.
The residual norms are calculated for each element separately as described in Section \ref{sec:adaptivity}, and the criteria for adaptivity is that those elements with local residual greater than one half of the maximum residual are refined in the directions perpendicular to $z$.
With these anisotropic adaptive refinements in place it is possible to apply the 2D results on point singularities from \cite{BabuskaAdaptive}, which imply that in the asymptotic limit the expected rate should be equivalent to that coming from a smooth solution.
That is, the rate with respect to $N_{\mathrm{dof}}$ is expected to approach $\frac{p}{3}$ in the limit.

\begin{figure}[!ht]
\begin{center}
\includegraphics[trim=3.4cm 0cm 1cm 0cm,clip=true,scale=0.39]{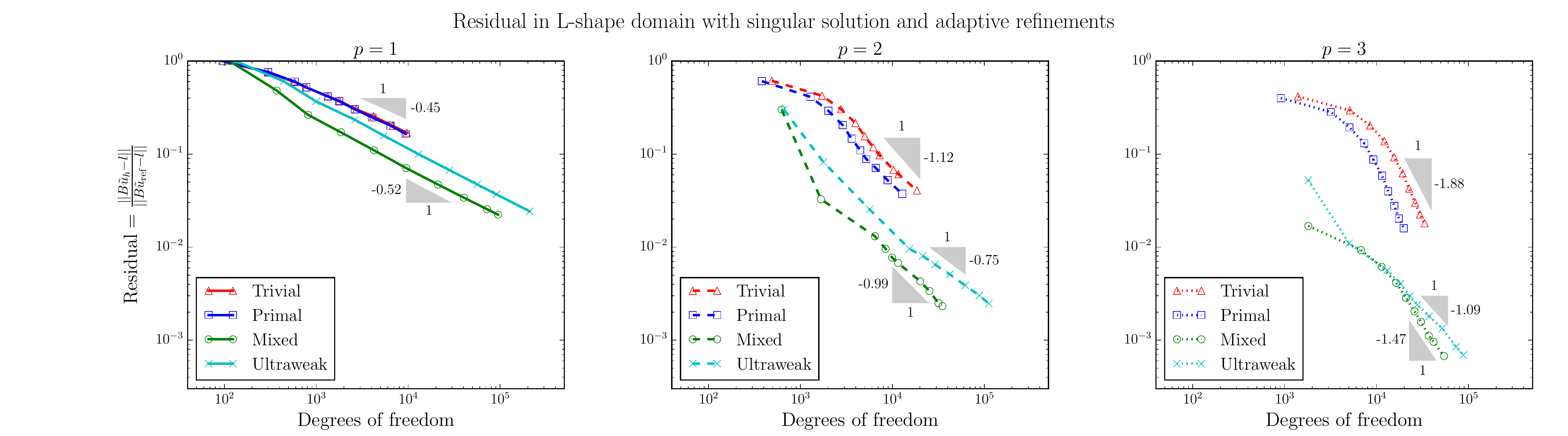}
\caption{Residual as a function of the degrees of freedom after adaptive anisotropic hexahedral refinements in the L-shape domain with a singular solution.}
\label{fig:Lshape-adaptive}
\end{center}
\end{figure}

The problem is solved successively through nine adaptive refinements with all formulations.
The results are illustrated in Figure \ref{fig:Lshape-adaptive}.
For $p=1$ the rates initially oscillate at around $0.5$, which is much better than the expected $0.33$.
This is a desirable quality, because the preasymptotic rates are faster than the expected rates.
Nevertheless, the rate would probably eventually approach the expected rate if more refinements had been taken.
Similar assertions hold for $p=2$ and $p=3$.
It is worth noting that the primal and strong formulations have very similar and consistent behaviors with respect to convergence.
On the other hand, for $p=2$ and $p=3$, the mixed and ultraweak formulations seem to have a less consistent behavior with adaptive refinements.

The adaptive refinement patterns for each of the different methods under this singular problem is interesting to analyze.
Indeed, note that for Figure \ref{fig:Lshape-adaptive} the mixed and ultraweak formulations evidence a greater growth in degrees of freedom with each adaptive step.
Figure \ref{fig:Lshape-adaptive-meshes} complements this by showing the resulting meshes for each of the methods after five refinements were performed.
As can be clearly seen, more elements have been refined with the mixed and ultraweak formulations than with the strong and primal formulations.
This is especially evident far from the re-entrant edge (where the singularity lies).
There could be many reasons for these refinement patterns, including the nature of the formulation itself and the choice of the test norm.
Indeed, the strong and primal formulations have the displacement variable, whose gradient is singular, lying in $\bfH^1$, while the two other formulations have it lying in $\bfL^2$.
This could imply that the residual is affected by those gradient terms, which leads to a much more focused pattern of refinements toward the singularity.
On the other hand, the choice of test norm is completely fundamental and can have a profound effect on the computations.
Here, we chose the standard norms.
However, other choice of norms, such as graph norms for the ultraweak formulation, might lead to radically different refinement patterns.

\begin{figure}[!ht]
\begin{center}
\includegraphics[scale=0.7]{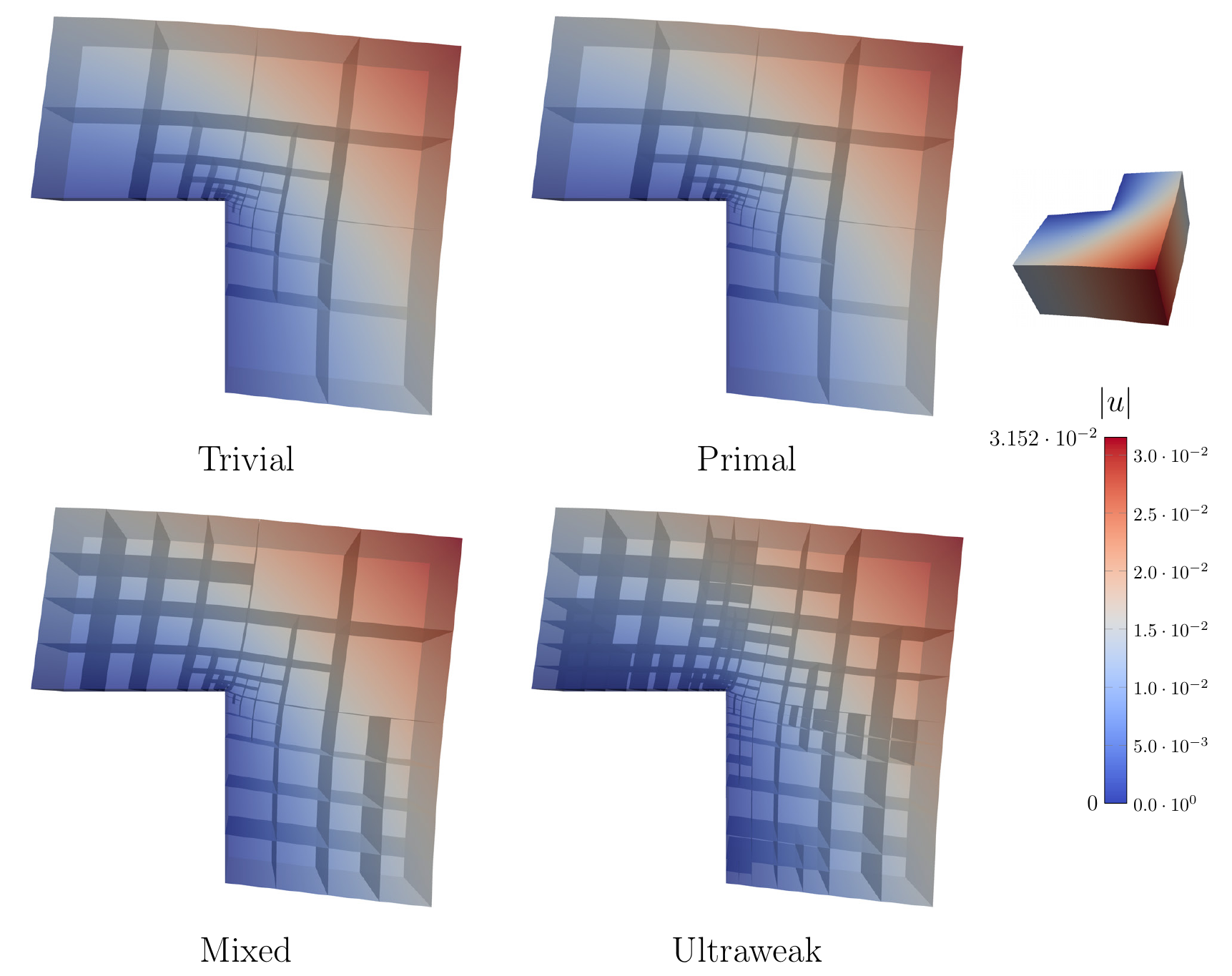}
\caption{The adaptive meshes for each method after five successive refinements. The domains are colored by the displacement magnitude, $|u|$, and warped by a factor of $10$.}
\label{fig:Lshape-adaptive-meshes}
\end{center}
\end{figure}



\section{Conclusions}
\label{sec:Conclusion}



This work was primarily a proof of concept of the DPG methodology in 3D linear elasticity.
We presented at least five different variational formulations of linear elasticity.
The formulations were exposed in both the traditional ``unbroken'' setting, and a setting with broken test spaces, which is suitable for the optimal test space DPG methodology to be applied.
The proofs of well-posedness in both settings were described.
In fact, all the ``unbroken'' formulations were proved to be mutually well-posed (see Appendix \ref{app:WellPosedness}), while the stability of the broken formulations followed from the unbroken case combined with a new theoretical framework detailed in \cite{Carstensen15} and carefully applied in this work.

Each of the formulations was numerically implemented using the DPG methodology.
In doing so, the applicability of the methodology was evidenced, since it was able to handle a wide array of variational formulations, including those where the test and trial spaces were completely different.
Moreover, a natural computation of the residual (in the context of DPG methods) was implemented for use in adaptive refinements.
The numerical results were in complete agreement with the theory, and the rates behaved as expected (or better) for different orders $p$ and with both smooth solutions \textit{and} singular solutions.
Interesting results were observed with singular solutions in relation to the adaptive refinement patterns produced by the different formulations.

In this paper, we made little attempt to speculate which formulation is better than the others.
The intention was only to show that it is viable to implement the same problem under the same methodology but with very different variational formulations.
However, that does not go without saying that some formulations may have strong advantages over others.
For example, some formulations are robustly stable in the incompressible limit while others are not.
A full comparison among formulations is a possible future endeavour.

The choices for test norms that we made were the standard norms and some future effort would also be appropriate to formulate better norms.
In particular, this could produce more desirable refinement patterns from adaptive schemes with the mixed and ultraweak formulations.

Another point of comment is that it is entirely feasible to solve a problem with different, yet compatible, variational formulations (such as the ones described in this work) in adjacent subdomains of the same domain. 
This is also left for future work.

%
%
%
%
%

\paragraph{Acknowledgements.}
The work of Keith, Fuentes, and Demkowicz was partially supported with grants by NSF (DMS-1418822),  AFOSR (FA9550-12-1-0484), and ONR (N00014-15-1-2496).

\phantomsection
\addcontentsline{toc}{section}{References}
\bibliographystyle{apalike}
\bibliography{main}

\appendix
\newpage
\section{Broken energy spaces: zero-jump lemma}
\label{app:zerojump}


The proof of Lemma \ref{lem:CharacterizationOfTraces} is presented here.

\begin{lemma}
Let $\,\Gamma_0$ and $\,\Gamma_1$ be relatively open subsets in $\bdry\Omega$ satisfying $\overline{\Gamma_0\cup\Gamma_1}=\bdry\Omega$ and $\Gamma_0\cap\Gamma_1=\varnothing$.
\begin{enumerate}[font=\upshape,label={(\roman*)},ref={\thelemma(\roman*)}]
	\item Let $v\in\bfH^1(\mesh)$. Then $v\in\bfH^1_{\Gamma_0}(\Omega)$ if and only if $\langle\hat{\tau}_\fkn,\tr_{\grad}v\rangle_{\bdry\mesh}=0$ for all $\hat{\tau}_\fkn\in\bfH^{-\frac{1}{2}}_{\Gamma_1}(\bdry\mesh)$. 
	\item Let $\tau\in\bfH(\div,\mesh)$. Then $\tau\in\bfH_{\Gamma_1}(\div,\Omega)$ if and only if $\langle\hat{u},\tr_{\div}\tau\rangle_{\bdry\mesh}=0$ for all $\hat{u}\in\bfH^{\frac{1}{2}}_{\Gamma_0}(\bdry\mesh)$. 
\end{enumerate}
\end{lemma}

\begin{proof}
We choose only to prove the first equivalence. The second equivalence is similar.

Let $v\in\bfH^1_{\Gamma_0}(\Omega)$ and $\hat{\tau}_\fkn\in\bfH^{-\frac{1}{2}}_{\Gamma_1}(\bdry\mesh)$.
By definition of $\bfH^{-\frac{1}{2}}_{\Gamma_1}(\bdry\mesh)$, there exists $\tau\in\bfH_{\Gamma_1}(\div,\Omega)$ such that $\tr_{\div}\tau=\hat{\tau}_\fkn$.
Given a domain $K$, for all $v\in\bfH^1(K)$ and $\tau\in\bfH(\div,K)$ it can be shown that the following distributional identity holds,
\begin{equation*}
  (\tau,\nabla v)_K+(\div\tau,v)_K=\langle\tr_{\div}^K\tau,\tr_{\grad}^Kv\rangle_{\bdry K}\,,
\end{equation*}
and in particular if $v\in\bfH^1_{\Gamma_0}(\Omega)$ and $\tau\in\bfH_{\Gamma_1}(\div,\Omega)$ the following identity holds,
\begin{equation*}
  (\tau,\nabla v)_\Omega+(\div\tau,v)_\Omega=\langle\tr_{\div}^\Omega\tau,\tr_{\grad}^\Omega v\rangle_{\bdry\Omega}=0\,.
\end{equation*}
Hence, rewriting the integral $(\tau,\nabla v)_\Omega+(\div\tau,v)_\Omega=0$ as a sum of integrals over each element in the mesh and using the first identity yields the result,
\begin{equation*}
	0=\sum_{K\in\mesh}(\tau,\nabla v)_K+(\div\tau,v)_K=\sum_{K\in\mesh}\langle\tr_{\div}^K\tau,\tr_{\grad}^Kv\rangle_{\bdry K}
		=\langle\hat{\tau}_\fkn,\tr_{\grad}v\rangle_{\bdry\mesh}\,.
\end{equation*}



For the converse assume $v\in\bfH^1(\mesh)$, so that $v|_K\in\bfH^1(K)$ for any $K\in\mesh$ and let $\hat{\tau}_\fkn\in\bfH^{-\frac{1}{2}}_{\Gamma_1}(\bdry\mesh)$, so that there exists $\tau\in\bfH_{\Gamma_1}(\div,\Omega)$ satisfying $\tr_{\div}\tau=\hat{\tau}_\fkn$.
Define $w$ such that $w|_K=\nabla (v|_K)$, meaning that $w\in\bfL^2(\Omega)$.
Then, using the hypothesis and the distributional identities gives,
\begin{equation*}
	0=\langle\hat{\tau}_\fkn,\tr_{\grad}v\rangle_{\bdry\mesh}=\sum_{K\in\mesh}(\tau,\nabla (v|_K))_K+(\div\tau,v|_K)_K
		=(\tau,w)_\Omega+(\div\tau,v)_\Omega\,.
\end{equation*}
In particular, for any smooth test function $\tau$, it holds that $(w,\tau)_\Omega=-(v,\div\tau)_\Omega$.
This means $w=\nabla v$ is the distributional derivative of $v$, so that $v\in\bfH^1(\Omega)$.
Next, let $\phi$ be a smooth test function defined on $\Gamma_0$ and with support in $\Gamma_0$, so that its zero extension to $\bdry\Omega$ satisfies that $\tr_{\div}^\Omega\tilde{\phi}\in\tr_{\div}^\Omega(\bfH_{\Gamma_1}(\div,\Omega))$, where $\tilde{\phi}\in\bfH_{\Gamma_1}(\div,\Omega)$ and $\tr_{\div}^\Omega\tilde{\phi}|_{\Gamma_0}=\phi$.
By definition of distributional restriction and the previous equality, it follows
\begin{equation*}
	\langle\phi,\tr_{\grad}^\Omega v|_{\Gamma_0}\rangle_{\Gamma_0}
 		=\langle\tr_{\div}^\Omega\tilde{\phi},\tr_{\grad}^\Omega v\rangle_{\bdry\Omega}
 			=(\tilde{\phi},\nabla v)_\Omega+(\div\tilde{\phi},v)_\Omega=0\,,
\end{equation*}
where the first distributional identity was utilized.
This is true for all smooth test functions $\phi$, implying $\tr_{\grad}^\Omega v|_{\Gamma_0}=0$, so that $v\in\bfH^1_{\Gamma_0}(\Omega)$.
\end{proof}

\section{Mutual well-posedness}
\label{app:WellPosedness}

The goal is to prove Theorem \ref{thm:mutuallywellposed}.
Throughout this section we assume $\Omega\subseteq\R^3$ is a three-dimensional bounded simply connected domain with a Lipschitz boundary $\bdry\Omega=\overline{\Gamma_0\cup\Gamma_1}$, where $\Gamma_0$ and $\Gamma_1$ are disjoint and relatively open in $\bdry\Omega$.
Note the results hold in two and one-dimensional domains as well.


Recall the variational formulations labeled as $(\scS)$, $(\scU)$, $(\scD)$, $(\scM)$ and $(\scP)$ (see \eqref{eq:TrivialFormulation}--\eqref{eq:PrimalFormulation}). 
The idea is to show these formulations are mutually ill or well-posed.
The concept of well-posedness is in the sense of Hadamard.
Well-posedness and stability estimates are proved using the well-known result by Babu\v{s}ka and Ne\v{c}as.

\begin{theorem}[Babu\v{s}ka-Ne\v{c}as]
Let $X$ and $Y$ be Hilbert spaces over a fixed field $\mathbb{F}\in\{\R,\mathbb{C}\}$, $\ell:Y\to\mathbb{F}$ be a continuous linear form and $b_0:X\times Y \to \mathbb{F}$ be a continuous bilinear form if $\,\mathbb{F}=\R$ or sesquilinear form if $\,\mathbb{F}=\mathbb{C}$.
If there exists an inf-sup constant $\gamma>0$ such that for all $x\in X$,
\begin{equation*}
	\sup_{y\in Y\setminus\{0\}} \frac{|b_0(x,y)|}{\|y\|_Y}\geq\gamma\|x\|_X\,,
\end{equation*}
and $\ell$ satisfies the compatibility condition
\begin{equation*}
	\ell(y)=0\,\text{ for all }\,y\in Y_{00}=\{y\in Y \mid b_0(x,y)=0 \,\text{ for all }\,x\in X\}\,,
\end{equation*}
then the problem
\begin{flalign*}
\qquad
 \left\{
  \begin{aligned}
   &\text{Find }\,x\in X,\\
   & b_0(x,y) = \ell(y)\,, \quad \text{for all } \,y\in Y\,,
  \end{aligned}
 \right. &&
\end{flalign*}
is well-posed, so that there exists a unique solution $x$ satisfying the stability estimate $\|x\|_X\leq\frac{1}{\gamma}\|\ell\|_{Y'}$.
\end{theorem}

For a given variational formulation $(\#)$, the corresponding bilinear form, linear form, spaces and constants are added a superscript $\#$. 
Indeed, $X^{\#}$, $Y^\#$ and $b_0^\#$ for the five variational formulations are
\begin{align}
	&\begin{aligned}
		&X^\scS=\bfH_{\Gamma_1}(\div,\Omega)\times\bfH^1_{\Gamma_0}(\Omega)\,,\qquad\qquad
			Y^\scS=\bfL^2(\Omega;\Sym)\times\bfL^2(\Omega)\times\bfL^2(\Omega;\Skw)\,,\\
		&b_0^\scS((\sigma,u),(\tau,v,w))=(\sigma,\tau)_\Omega-(\C:\nabla u,\tau)_\Omega-(\div\sigma,v)_\Omega+(\sigma,w)_\Omega\,,
	\end{aligned}\displaybreak[2]\\[2mm]
	&\begin{aligned}
		&X^\scU=\bfL^2(\Omega;\Sym)\times\bfL^2(\Omega)\times\bfL^2(\Omega;\Skw)\,,\qquad\qquad
			Y^\scU=\bfH_{\Gamma_1}(\div,\Omega)\times\bfH^1_{\Gamma_0}(\Omega)\,,\\
		&b_0^\scU((\sigma,u,\omega),(\tau,v))=(\A:\sigma,\tau)_\Omega+(\omega,\tau)_\Omega+(u,\div\tau)_\Omega+(\sigma,\nabla v)_\Omega\,,
	\end{aligned}\displaybreak[2]\\[2mm]
	&\begin{aligned}
		&X^\scD=\bfL^2(\Omega;\Sym)\times\bfH^1_{\Gamma_0}(\Omega)\,,\qquad\qquad
			Y^\scD=\bfL^2(\Omega;\Sym)\times\bfH^1_{\Gamma_0}(\Omega)\,,\\
		&b_0^\scD((\sigma,u),(\tau,v))=(\sigma,\tau)_\Omega-(\C:\nabla u,\tau)_\Omega+(\sigma,\nabla v)_\Omega\,,
	\end{aligned}\displaybreak[2]\\[2mm]
	&\begin{aligned}
		&X^\scM=\bfH_{\Gamma_1}(\div,\Omega)\times\bfL^2(\Omega)\times\bfL^2(\Omega;\Skw)\,,\qquad\qquad
			Y^\scM=\bfH_{\Gamma_1}(\div,\Omega)\times\bfL^2(\Omega)\times\bfL^2(\Omega;\Skw)\,,\\
		&b_0^\scM((\sigma,u,\omega),(\tau,v,w))=(\A:\sigma,\tau)_\Omega+(\omega,\tau)_\Omega+(u,\div\tau)_\Omega-(\div\sigma,v)_\Omega
			+(\sigma,w)_\Omega\,,
	\end{aligned}\displaybreak[2]\\[2mm]
	&\begin{aligned}
		&X^\scP=\bfH^1_{\Gamma_0}(\Omega)\,,\qquad\qquad
			Y^\scP=\bfH^1_{\Gamma_0}(\Omega)\,,\\
		&b_0^\scP(u,v)=(\C:\nabla u,\nabla v)_\Omega\,.
	\end{aligned}	
\end{align}

Additionally, consider $(\scS_\Sym)$, $(\scU_\Sym)$ and $(\scM_\Sym)$, which are new variational formulations using the space $\bfH_{\Gamma_1}(\div,\Omega;\Sym)$ as opposed to $\bfH_{\Gamma_1}(\div,\Omega)$.
Their defining spaces and forms are
\begin{align}
	&\begin{aligned}
		&X^{\scS_\Sym}=\bfH_{\Gamma_1}(\div,\Omega;\Sym)\times\bfH^1_{\Gamma_0}(\Omega)\,,\qquad\qquad
			Y^{\scS_\Sym}=\bfL^2(\Omega;\Sym)\times\bfL^2(\Omega)\,,\\
		&b_0^{\scS_\Sym}((\sigma,u),(\tau,v))=(\sigma,\tau)_\Omega-(\C:\nabla u,\tau)_\Omega-(\div\sigma,v)_\Omega\,,\qquad
			\ell^{\scS_\Sym}((\tau,v))=(f,v)_\Omega\,,
	\end{aligned}\displaybreak[2]\\[2mm]
	&\begin{aligned}
		&X^{\scU_\Sym}=\bfL^2(\Omega;\Sym)\times\bfL^2(\Omega)\,,\qquad\qquad\qquad\quad\,
			Y^{\scU_\Sym}=\bfH_{\Gamma_1}(\div,\Omega;\Sym)\times\bfH^1_{\Gamma_0}(\Omega)\,,\\
		&b_0^{\scU_\Sym}((\sigma,u),(\tau,v))=(\A:\sigma,\tau)_\Omega+(u,\div\tau)_\Omega+(\sigma,\nabla v)_\Omega\,,\qquad
			\ell^{\scU_\Sym}((\tau,v))=(f,v)_\Omega\,,
	\end{aligned}\displaybreak[2]\\[2mm]
	&\begin{aligned}
		&X^{\scM_\Sym}=\bfH_{\Gamma_1}(\div,\Omega;\Sym)\times\bfL^2(\Omega)\,,\qquad\quad\quad\,\,
			Y^{\scM_\Sym}=\bfH_{\Gamma_1}(\div,\Omega;\Sym)\times\bfL^2(\Omega)\,,\\
		&b_0^{\scM_\Sym}((\sigma,u),(\tau,v))=(\A:\sigma,\tau)_\Omega+(u,\div\tau)_\Omega-(\div\sigma,v)_\Omega\,,\quad
			\ell^{\scM_\Sym}((\tau,v))=(f,v)_\Omega\,.
	\end{aligned}
\end{align}


The proof of mutual well-posedness is discussed in two parts.
First, the mutual satisfaction of the compatibility conditions is analyzed.
Second, the inf-sup constants are also shown to be mutually satisfied.

Throughout, note that the proofs only hold in the compressible regime.
Here, $\C$ and $\A$ are inverse to each other over $\Sym$.
This is no longer true in the incompressible case (in the limit of $\lambda\to\infty$), where only the variational formulations that make use of $\A$ can be proved to remain well-posed.

\subsection{Compatibility conditions}

Well-posedness of the variational formulations depends on the nature of $\Gamma_0$ and $\Gamma_1$.
The first lemma shows that $\Gamma_0\neq\varnothing$ is a necessary condition for all variational formulations to be well-posed.
The condition is also sufficient, and this is the content of Corollary \ref{cor:allwellposed}.

\begin{lemma}
\label{lem:Gammazeronotempty}
Suppose one of the variational formulations among $(\scS)$, $(\scU)$, $(\scD)$, $(\scM)$, $(\scP)$, $(\scS_\Sym)$, $(\scU_\Sym)$ and $(\scM_\Sym)$ is well-posed.
Then $\Gamma_0\neq\varnothing$.
\end{lemma}

\begin{proof}
Assume the hypothesis so that the well-posed variational formulation has a unique solution $x$, whose component $u$ is the displacement solution variable.
By contradiction assume $\Gamma_0=\varnothing$.
Then any translation (constant) $u_C$ satisfies the boundary conditions vacuously and $\nabla u_C=0$.
For the variational formulations  $(\scS)$, $(\scS_\Sym)$, $(\scD)$ and $(\scP)$ it is straightforward that, ceteris paribus, the solution $x_C$ with displacement component $u+u_C$ is a different solution (provided $u_C\neq0$) to the original problem.
Similarly, since $u_C\in\bfH^1_{\Gamma_0}(\Omega)$ and $\nabla u_C=0$, the distributional identity yields $(u_C,\div\tau)_\Omega=-(\nabla u_C,\tau)_\Omega=0$ for all $\tau\in\bfH_{\Gamma_1}(\div,\Omega)$, so that $x_C$ is also a different solution to the variational formulations $(\scU)$, $(\scU_\Sym)$, $(\scM)$ and $(\scM_\Sym)$.
This contradicts that the original solution was unique.
\end{proof}

The next lemma shows that the solution to the original elasticity equation, \eqref{eq:LinElast}, with homogeneous forcing and boundary conditions ($f=0$, $u_0=0$ and $g=0$) is $u=0$ and is unique provided $\Gamma_0\neq\varnothing$.

\begin{lemma}
\label{lem:homogeneousuniquesolution}
Suppose $\Gamma_0\neq\varnothing$ and consider the equation $-\div(\C:\varepsilon(u))=0$ in $\Omega$, where $u$ is sought in $\bfH^1_{\Gamma_0}(\Omega)$ and $\C:\varepsilon(u)\in\bfH_{\Gamma_1}(\div,\Omega)$. 
Then $u=0$ is the unique solution to the problem.
\end{lemma}

\begin{proof}
Multiplying the equation by a test function $v\in\bfH^1_{\Gamma_0}(\Omega)$, integrating and using a distributional identity yields the equation $\int_\Omega \nabla u : \C : \nabla v \dd\Omega=0$ for all $v\in\bfH^1_{\Gamma_0}(\Omega)$, which is precisely the formulation $(\scP)$ with $f=0$.
Using Korn's inequality and that $\Gamma_0\neq\varnothing$, the bilinear form can be shown to be coercive, meaning $b_0^\scP(u,u)=\int_\Omega \nabla u : \C : \nabla u \dd\Omega\geq\alpha\|u\|_{\bfH^1(\Omega)}^2$ for some $\alpha>0$ \cite{Ciarlet13}.
Taking $v=u$, the equation becomes $b_0^\scP(u,u)=0$, and using coercivity it implies $\|u\|_{\bfH^1(\Omega)}=0$, so that $u=0$ is the only solution.
\end{proof}

Finally, it is shown that given $\Gamma_0\neq\varnothing$, the compatibility condition is satisfied trivially for every variational formulation.

\begin{lemma}
\label{lem:trivialcompatibility}
Let $\Gamma_0\neq\varnothing$.
Then the variational formulations $(\scS)$, $(\scU)$, $(\scD)$, $(\scM)$, $(\scP)$, $(\scS_\Sym)$, $(\scU_\Sym)$ and $(\scM_\Sym)$ all have a trivial compatibility space, implying that the compatibility conditions are satisfied trivially for any linear form. 
\end{lemma}

\begin{proof}
First consider $(\scS)$.
The aim is to prove $Y_{00}^\scS=\{0\}$.
Let $x=(\sigma,u)\in X^\scS$, with $u=0$ and $\sigma$ being any smooth symmetric matrix field vanishing at the boundary.
The condition $b_0^\scS(x,y)=0$ then becomes $\int_\Omega\tau:\sigma-v\cdot\div\sigma\dd\Omega=0$, which yields the distributional equality $-\varepsilon(v)=\tau\in\bfL^2(\Omega;\Sym)$.
By Korn's inequality, $v\in\bfH^1(\Omega)$, and further testing against $\sigma\in\bfH_{\Gamma_1}(\div,\Omega;\Sym)$ yields additionally that $v\in\bfH^1_{\Gamma_0}(\Omega)$.
Next, test with $\sigma=0$ and $u\in\bfH^1_{\Gamma_0}(\Omega)$, so that $b_0^\scS(x,y)=0$ yields $\int_\Omega\nabla u:\C:\nabla v\dd\Omega=0$, which can be rewritten as $-\div(\C:\varepsilon(v))=0$.
By Lemma \ref{lem:homogeneousuniquesolution}, $v=0$, meaning $\tau=-\varepsilon(v)=0$.
Finally, $b_0^\scS(x,y)=0$ becomes $\int_\Omega \sigma:w\dd\Omega=0$ when testing with $\sigma\in\bfH_{\Gamma_1}(\div,\Omega)$ (nonsymmetric), which results in $w=0$ as well.
Therefore $y=(0,0,0)$ is the only element of $Y_{00}^\scS$.

Next consider $(\scU)$ and the condition $b_0^\scU(x,y)=0$ for all $x=(u,\sigma,\omega)\in X^\scU$.
First let $\sigma=0$ and $u=0$, so that the condition becomes $\int_\Omega\omega:\tau\dd\Omega=0$.
Therefore, the antisymmetric part of $\tau$ vanishes, meaning $\tau\in\bfH_{\Gamma_1}(\div,\Omega;\Sym)$.
Then, with $\sigma=0$, the condition becomes $\int_\Omega u\cdot\div\tau\dd\Omega=0$, so that $\div\tau=0$.
Finally, test with $u=0$, so that the condition yields the equation $\A:\tau+\varepsilon(v)=0$, which can be rewritten as $\tau=-\C:\varepsilon(v)$.
Taking the divergence and using $\div\tau=0$ gives $-\div(\C:\varepsilon(v))=0$, which by Lemma \ref{lem:homogeneousuniquesolution} results in $v=0$ and $\tau=-\C:\varepsilon(v)=0$.
Hence, $Y_{00}^\scU=\{0\}$.

Similar calculations follow for $(\scD)$, $(\scM)$, $(\scP)$, $(\scS_\Sym)$, $(\scU_\Sym)$ and $(\scM_\Sym)$. 
\end{proof}

When $\Gamma_0=\varnothing$ it is possible to redefine some spaces by considering the quotient over a particular null space (e.g.~rigid body motions).
This essentially produces new closely related, yet modified variational formulations which are well-posed even when $\Gamma_0=\varnothing$.
Indeed, after redefining these spaces, a relevant version of Korn's inequality can be proved to hold \cite[Theorem 2.3]{CiarletKorns10}. 
However, in this work we chose not to redefine those spaces.

\subsection{Inf-sup constants}

Before proceeding to the main result, the most challenging results are proved as three independent lemmas.
The closed range theorem in the closed operator setting plays a key role in two of these lemmas, while the remaining lemma uses the Rellich-Kondrachov theorem to prove a relevant Poincar\'{e}-type inequality.

\begin{lemma}
\label{lem:closedrangesymmetric}
The formulations $(\scS_\Sym)$ and $(\scU_\Sym)$ are mutually ill or well-posed.
\end{lemma}

\begin{proof}
Assume $(\scS_\Sym)$ is well-posed, so the compatibility conditions are satisfied and $\gamma^{\scS_\Sym}>0$ exists.
Then by Lemma \ref{lem:Gammazeronotempty}, $\Gamma_0\neq\varnothing$.
Using Lemma \ref{lem:trivialcompatibility}, it follows $Y_{00}^{\scS_\Sym}=\{0\}$ and $Y_{00}^{\scU_\Sym}=\{0\}$ so the compatibility conditions are satisfied for $(\scS_\Sym)$ and $(\scU_\Sym)$.
It remains to show the existence of $\gamma^{\scU_\Sym}>0$.

The first step is to recognize the underlying linear operators in $(\scS_\Sym)$ and $(\scU_\Sym)$.
Indeed, for $x=(\sigma,u)$ and $y=(\tau,v)$,
\begin{equation*}
	b_0^{\scS_\Sym}(x,y)=(A_\scS x,y)_\Omega\,,\qquad\qquad b_0^{\scU_\Sym}(x,y)=(x,A_\scU y)_\Omega\,,
\end{equation*}
where
\begin{gather*}
	A_\scS:X^{\scS_\Sym}\rightarrow Y^{\scS_\Sym}\,,\qquad\qquad\qquad A_\scU:Y^{\scU_\Sym}\rightarrow X^{\scU_\Sym}\,,\\
	A_\scS x=\begin{pmatrix}I & -\C:\varepsilon \\ -\div & 0 \end{pmatrix} \begin{pmatrix} \sigma \\ u \end{pmatrix}\,,\qquad\qquad
	A_\scU y=\begin{pmatrix}\A & \varepsilon \\ \div & 0 \end{pmatrix} \begin{pmatrix} \tau \\ v \end{pmatrix}\,.
\end{gather*}
Define $L=Y^{\scS_\Sym}=X^{\scU_\Sym}=\bfL^2(\Omega;\Sym)\times\bfL^2(\Omega)$ and $\mD=X^{\scS_\Sym}=Y^{\scU_\Sym}=\bfH_{\Gamma_1}(\div,\Omega;\Sym)\times\bfH^1_{\Gamma_0}(\Omega)$.
In the topology of $L$, $\mD$ is dense in $L$, and $A_\scS$ and $A_\scU$ are well-defined closed operators.
Meanwhile, if $\mD$ is suited with a graph norm such as $\|x\|_{A_\scS}^2=\|x\|_{L}^2+\|A_\scS x\|_{L}^2$ or $\|y\|_{A_\scU}^2=\|y\|_{L}^2+\|A_\scU y\|_{L}^2$ or with the standard norm $\|(\sigma,u)\|_\mD^2=\|\sigma\|_{\bfH(\div,\Omega)}^2+\|u\|_{\bfH^1(\Omega)}^2$, the operators $A_\scS$ and $A_\scU$ are well-defined continuous operators.
It can be shown that the norms $\|\cdot\|_{A_\scS}$, $\|\cdot\|_{A_\scU}$ and $\|\cdot\|_\mD$ are equivalent.

Moreover, both operators are injective.
Indeed, suppose $A_\scS x=0$ and $A_\scU y=0$ for $x=(\sigma,u)$ and $y=(\tau,v)$, so that $\sigma-\C:\varepsilon(u)=0$, $-\div\sigma=0$, $\A:\tau+\varepsilon(v)=0$ and $\div\tau=0$.
These can be rewritten as $-\div(\C:\varepsilon(u))=0$ and $-\div(\C:\varepsilon(v))=0$ respectively, and, since $\Gamma_0\neq\varnothing$, by Lemma \ref{lem:homogeneousuniquesolution}, $u=v=0$ and $\sigma=\tau=0$, so that $x=y=0$.

Next, notice the closed operator adjoints of $A_\scS$ and $A_\scU$ are closely related.
In fact, $A_\scS^*=A_\scU\mathsf{M}$ and $A_\scU^*=\mathsf{M}^{-1}A_\scS$, where $\mathsf{M}=(\begin{smallmatrix} \C & 0 \\ 0 & I \end{smallmatrix}):L\to L$ is an invertible bounded linear operator with continuous inverse $\mathsf{M}^{-1}=(\begin{smallmatrix} \A & 0 \\ 0 & I \end{smallmatrix}):L\to L$.
Therefore, both $A_\scS^*$ and $A_\scU^*$ are injective.

Now, using that $(\scS_\Sym)$ is well-posed, it follows there exists $\gamma^{\scS_\Sym}>0$ such that for all $x\in\mD$
\begin{equation*}
	\|A_\scS x\|_L=\sup_{y\in L\setminus\{0\}}\frac{|(A_\scS x,y)_\Omega|}{\|y\|_L}
		=\sup_{y\in L\setminus\{0\}}\frac{|b_0^{\scS_\Sym}(x,y)|}{\|y\|_L}\geq\gamma^{\scS_\Sym}\|x\|_{\mD}\geq\gamma^{\scS_\Sym}\|x\|_L\,.
\end{equation*}
Using the closed range theorem for closed operators along with the injectivity of $A_\scS$ and $A_\scS^*$, it follows that $A_\scS$ and $A_\scS^*$ are surjective, so that $A_\scS$ and $A_\scU=A_\scS^*\mathsf{M}^{-1}$ are bijective, and for all $y\in\mD$
\begin{equation*}
	\|A_\scU y\|_L=\|A_\scS^*\mathsf{M}^{-1}y\|_L\geq\gamma^{\scS_\Sym}\|\mathsf{M}^{-1}y\|_L
		\geq\textstyle{\frac{\gamma^{\scS_\Sym}}{\|\mathsf{M}\|}}\|y\|_L\,,
\end{equation*}
where it was used that $\|y\|_L\leq\|\mathsf{M}\|\|\mathsf{M}^{-1}y\|_L$, where $\|\mathsf{M}\|$ is the operator norm of $\mathsf{M}$. 
Squaring the inequality and adding $C_\scS^2\|A_\scU y\|_L^2$ on both sides, where $C_\scS=\frac{\gamma^{\scS_\Sym}}{\|\mathsf{M}\|}$, yields for all $y\in\mD$
\begin{equation*}
	\|A_\scU y\|_L\geq\sqrt{\textstyle{\frac{C_\scS^2}{1+C_\scS^2}}}\|y\|_{A_\scU}
		\geq C_{\scU\mD}\sqrt{\textstyle{\frac{C_\scS^2}{1+C_\scS^2}}}\|y\|_\mD\,,
\end{equation*}
where $C_{\scU\mD}$ is the relevant equivalence constant between the norms $\|\cdot\|_{A_\scU}$ and $\|\cdot\|_\mD$.
Let $\gamma^{\scU_\Sym}>0$ be defined by $(\gamma^{\scU_\Sym})^2=\frac{C_{\scU\mD}^2C_\scS^2}{1+C_\scS^2}$.
Since $A_\scU:\mD\to L$ is bijective, it follows it is invertible with inverse $A_\scU^{-1}:L\to\mD$, which is continuous (by the open mapping theorem) when $\mD$ is viewed as a normed space.
The continuous operator transpose $(A_\scU^{-1})':\mD'\to L'=L$ therefore exists, and by its properties it follows its operator norm is $\|(A_\scU^{-1})'\|=\|A_\scU^{-1}\|$.
Moreover, $(A_\scU^{-1})'=(A_\scU')^{-1}$ where $A_\scU':L'=L\to\mD'$ is the continuous operator transpose of $A_\scU$ satisfying $(x,A_\scU y)_\Omega=\langle A_\scU'x,y\rangle_{\mD'\times\mD}$ for all $x\in L$ and $y\in\mD$.
Hence,
\begin{equation*}
	\begin{aligned}
		\gamma^{\scU_\Sym}\leq\inf_{y\in\mD\setminus\{0\}}&\frac{\|A_\scU y\|_L}{\|y\|_\mD}
			=\Big(\sup_{y\in\mD\setminus\{0\}}\frac{\|y\|_\mD}{\|A_\scU y\|_L}\Big)^{-1}
				=\Big(\sup_{x\in L\setminus\{0\}}\frac{\|A_\scU^{-1} x\|_\mD}{\|x\|_L}\Big)^{-1}=\frac{1}{\|A_\scU^{-1}\|}\\
					&=\frac{1}{\|(A_\scU^{-1})'\|}=\frac{1}{\|(A_\scU')^{-1}\|}=\inf_{x\in L\setminus\{0\}}\frac{\|A_\scU' x\|_{\mD'}}{\|x\|_L}
						=\inf_{x\in L\setminus\{0\}}\sup_{y\in\mD\setminus\{0\}}
							\frac{|\langle A_\scU'x,y\rangle_{\mD'\times\mD}|}{\|x\|_L\|y\|_\mD}\\
								&=\inf_{x\in L\setminus\{0\}}\sup_{y\in\mD\setminus\{0\}}\frac{|(x,A_\scU y)_\Omega|}{\|x\|_L\|y\|_\mD}
									=\inf_{x\in L\setminus\{0\}}\sup_{y\in\mD\setminus\{0\}}\frac{|b_0^{\scU_\Sym}(x,y)|}{\|x\|_L\|y\|_\mD}\,.
	\end{aligned}
\end{equation*}
This shows the existence of $\gamma^{\scU_\Sym}>0$ satisfying the desired property, meaning $(\scU_\Sym)$ is well-posed.

Similar calculations show that if $(\scU_\Sym)$ is well-posed then $(\scS_\Sym)$ is well-posed.
\end{proof}


\begin{lemma}
\label{lem:poincarelike}
Let $\Gamma_0\neq\varnothing$.
There exists a constant $C_P>0$ such that for all $u\in\bfH^1_{\Gamma_0}(\Omega)$ and $\omega\in\bfL^2(\Omega;\Skw)$,
\begin{equation*}
	\sqrt{\|u\|_{\bfL^2(\Omega)}^2+\|\omega\|_{\bfL^2(\Omega;\Skw)}^2}\leq C_P\|-\nabla u+\omega\|_{\bfL^2(\Omega;\Mat)}\,.
\end{equation*}
\end{lemma}

\begin{proof}
Let $L=\bfL^2(\Omega)\times\bfL^2(\Omega;\Skw)$ and $\|\cdot\|_L$ be its Hilbert norm.
Suppose by contradiction that such constant $C_P$ does not exist.
Then, for every $n\in\N$ there exists $(\tilde{u}_n,\tilde{\omega}_n)\in\bfH^1_{\Gamma_0}(\Omega)\times\bfL^2(\Omega;\Skw)$ such that
\begin{equation*} 
	\|(\tilde{u}_n,\tilde{\omega}_n)\|_L>n\|-\nabla \tilde{u}_n+\tilde{\omega}_n\|_{\bfL^2(\Omega;\Mat)}\,.
\end{equation*}
Let $(u_n,\omega_n)=\frac{1}{\|(\tilde{u}_n,\tilde{\omega}_n)\|_L}(\tilde{u}_n,\tilde{\omega}_n)$ so that $\|(u_n,\omega_n)\|_L=1$ and $\|-\nabla u_n+\omega_n\|_{\bfL^2(\Omega;\Mat)}<\frac{1}{n}$ for all $n\in\N$.
Note $(\omega_n)_{n\in\N}\subseteq\bfL^2(\Omega;\Skw)$ is antisymmetric so taking the symmetric part of the previous inequality yields $\|\varepsilon(u_n)\|_{\bfL^2(\Omega;\Sym)}\leq\|-\nabla u_n+\omega_n\|_{\bfL^2(\Omega;\Mat)}<\frac{1}{n}$ for all $n\in\N$.
Moreover, $\|\omega_n\|_{\bfL^2(\Omega;\Skw)}\leq\|(u_n,\omega_n)\|_L=1$, $\|u_n\|_{\bfL^2(\Omega)}\leq1$ and $\|\nabla u_n\|_{\bfL^2(\Omega;\Mat)}\leq\|-\nabla u_n+\omega_n\|_{\bfL^2(\Omega;\Mat)}+\|\omega_n\|_{\bfL^2(\Omega;\Skw)}\leq2$, so $\|u_n\|_{\bfH^1(\Omega)}\leq\sqrt{5}$ for all $n\in\N$ and by the Rellich-Kondrachov theorem it follows there exists a subsequence convergent to some $u$ in $\bfL^2(\Omega)$, $\lim_{k\to\infty}\|u_{n_k}-u\|_{\bfL^2(\Omega)}=0$.
Then $\varepsilon(u_{n_k})$ converges to $\varepsilon(u)$ as distributions, which in turn implies $\varepsilon(u)=0$.
Thus, $-\div(\C:\varepsilon(u))=0$ and by Lemma \ref{lem:homogeneousuniquesolution} it follows $u=0$.
Using Korn's inequality yields $\lim_{k\to\infty}\|u_{n_k}\|_{\bfH^1(\Omega)}=0$, so that in particular $(\nabla u_{n_k})_{k\in\N}$  converges to $\nabla u=0$ in $\bfL^2(\Omega;\Mat)$ and as a result $(\omega_{n_k})_{k\in\N}$ converges to $\omega=0$ in $\bfL^2(\Omega;\Skw)$ as well.
Recalling that $\|(u_{n_k},\omega_{n_k})\|_L=1$, leads to $\|(u_{n_k},\omega_{n_k})-(u,\omega)\|_L\geq|\|(u_{n_k},\omega_{n_k})\|_L-\|(u,\omega)\|_L|=1$ for all $k\in\N$ and this contradicts that $(u_{n_k},\omega_{n_k})_{k\in\N}$ is convergent to $(u,\omega)=(0,0)$ in $L$.
\end{proof}

The next lemma proves an inf-sup condition which is the same as one of the Brezzi conditions for $(\scM)$ \cite{mixedelas3d}.
It presents an alternate proof to that provided in \cite{arnold2006finite,falk2008} and uses the closed range theorem as opposed to differential forms.

\begin{lemma}
\label{lem:brezzicondition}
Let $\Gamma_0\neq\varnothing$.
There exists a constant $C_B>0$ such that for all $u\in\bfL^2(\Omega)$ and $\omega\in\bfL^2(\Omega;\Skw)$,
\begin{equation*}
	C_B\sqrt{\|u\|_{\bfL^2(\Omega)}^2+\|\omega\|_{\bfL^2(\Omega;\Skw)}^2}
		\leq \sup_{\tau\in\bfH_{\Gamma_1}(\div,\Omega)\setminus\{0\}}
			\frac{|(u,\div\tau)_\Omega+(\omega,\tau)_\Omega|}{\|\tau\|_{\bfH(\div,\Omega)}}\,.
\end{equation*}
\end{lemma}

\begin{proof}
The proof is very similar to that of Lemma \ref{lem:closedrangesymmetric}.
First consider
\begin{gather*}
	A_\scW:\mD_{\scW}\rightarrow L_\Mat\,,\qquad\qquad\qquad A_\scV:\mD_\scV\rightarrow L\,,\\
	A_\scW(u,\omega)=-\nabla u+\omega\,,\qquad\qquad
	A_\scV\tau=(\div\tau,\textstyle{\frac{1}{2}}(\tau-\tau^\T))\,,
\end{gather*}
where $\mD_\scW=\bfH^1_{\Gamma_0}(\Omega)\times\bfL^2(\Omega;\Skw)$, $\mD_\scV=\bfH_{\Gamma_1}(\div,\Omega)$, $L_\Mat=\bfL^2(\Omega;\Mat)$ and $L=\bfL^2(\Omega)\times\bfL^2(\Omega;\Skw)$.
Clearly in the topologies of $L$ and $L_\Mat$, the domains $\mD_{\scW}$ and $\mD_\scV$ are dense in $L$ and $L_\Mat$ respectively.
With these topologies $A_\scW$ and $A_\scV$ are well-defined closed operators.
If $\mD_\scW$ is endowed with the graph norm $\|(u,\omega)\|_{A_\scW}^2=\|(u,\omega)\|_{L}^2+\|A_\scW(u,\omega)\|_{L_\Mat}^2$ or with the standard norm $\|(u,\omega)\|_{\mD_\scW}^2=\|u\|_{\bfH^1(\Omega)}^2+\|\omega\|_{L_\Mat}^2$, then $A_\scW$ is a well-defined continuous operator.
Note $\|\cdot\|_{A_\scW}$ and $\|\cdot\|_{\mD_\scW}$ are equivalent norms.
Similarly, if $\mD_\scV$ is given the graph norm $\|\tau\|_{A_\scV}^2=\|\tau\|_{L_\Mat}^2+\|A_\scV\tau\|_L^2$ or the standard norm $\|\tau\|_{\mD_\scV}=\|\tau\|_{\bfH(\div,\Omega)}$ then $A_\scV$ is a well-defined continuous operator.
Note $\|\cdot\|_{A_\scV}$ and $\|\cdot\|_{\mD_\scV}$ are equivalent norms.

As closed operators, $A_\scW$ and $A_\scV$ are clearly adjoint to each other, so that $A_\scW^*=A_\scV$.
Moreover, if $A_\scW(u,\omega)=0$, then $\nabla u=\omega\in\bfL^2(\Omega;\Skw)$, so that $\varepsilon(u)=0$.
This implies $-\div(\C:\varepsilon(u))=0$ and by Lemma \ref{lem:homogeneousuniquesolution}, $u=0$ and $\omega=\nabla u=0$, so that $A_\scW$ is injective.
On the other hand, $A_\scV$ has a nontrivial null space, because if $A_\scV\tau=(0,0)$, then $\tau\in\sfN(A_\scV)=\{\tau_0\in\bfH_{\Gamma_1}(\div,\Omega;\Sym)\mid\div\tau_0=0\}$.

By Lemma \ref{lem:poincarelike}, it follows that for all $(u,\omega)\in\mD_{\scW}$, $\|A_\scW(u,\omega)\|_{L_\Mat}\geq\frac{1}{C_P}\|(u,\omega)\|_L$.
By the closed range theorem, $\|\tilde{A}_\scV[\tau]\|_L=\|A_\scV\tau\|_L\geq\frac{1}{C_P}\|[\tau]\|_{L_\Mat/\sfN(A_\scV)}$ for all $\tau\in\mD_\scV$, where $\tilde{A}_\scV:\tilde{\mD}_\scV\to L$ is defined by $\tilde{A}_\scV[\tau]=A_\scV\tau$, with $\tilde{\mD}_\scV=\mD_\scV/\sfN(A_\scV)$ and $\|[\tau]\|_{L_\Mat/\sfN(A_\scV)}=\inf_{\tau_0\in\sfN(A_\scV)}\|\tau+\tau_0\|_{L_\Mat}$.
Then, $\|\tilde{A}_\scV[\tau]\|_L\geq C_B\|[\tau]\|_{\tilde{\mD}_\scV}$ for all $\tau\in\mD_\scV$, where $\|[\tau]\|_{\tilde{\mD}_\scV}=\inf_{\tau_0\in\sfN(A_\scV)}\|\tau+\tau_0\|_{\mD_\scV}$, $C_B=C_{\scV\mD}\frac{1}{1+C_P^2}$ and $C_{\scV\mD}$ is the relevant equivalence constant between $\|\cdot\|_{A_\scV}$ and $\|\cdot\|_{\mD_\scV}$.
The closed range theorem also implies $\sfR(\tilde{A}_\scV)=\sfR(A_\scV)=L$ is closed, so that $\tilde{A}_\scV$ is bijective and by the open mapping theorem it is a homeomorphism with continuous inverse $\tilde{A}_\scV^{-1}:L\to\tilde{\mD}_\scV$.
Using the continuous operator transpose of $\tilde{A}_\scV$ and $\tilde{A}_\scV^{-1}$ as in the proof of Lemma \ref{lem:closedrangesymmetric} yields for all $x=(u,\omega)\in L$,
\begin{align*}
	\begin{aligned}
		C_B\|x\|_L
			&\leq\sup_{[\tau]\in\tilde{\mD}_\scV\setminus\{[0]\}}\frac{|(x,\tilde{A}_\scV[\tau])_\Omega|}{\|[\tau]\|_{\tilde{\mD}_\scV}}
				=\sup_{\tau^\perp\in Z\setminus\{0\}}\sup_{\tau_0\in\sfN(A_\scV)}
					\frac{|(x,A_\scV\tau^\perp)_\Omega|}{\inf_{\tau_0'\in\sfN(A_\scV)}\|\tau^\perp+\tau_0+\tau_0'\|_{\mD_\scV}}\\
			&=\sup_{\tau^\perp\in Z\setminus\{0\}}\sup_{\tau_0\in\sfN(A_\scV)}\sup_{\tau_0'\in\sfN(A_\scV)}
				\frac{|(x,A_\scV\tau^\perp)_\Omega|}{\|\tau^\perp+\tau_0+\tau_0'\|_{\mD_\scV}}
					=\sup_{\tau\in\mD_\scV\setminus\{0\}}\frac{|(x,A_\scV\tau)_\Omega|}{\|\tau\|_{\mD_\scV}}\,,
	\end{aligned}
\end{align*}
where $Z$ is any algebraic complement to $\sfN(A_\scV)$ so that $\mD_\scV=\sfN(A_\scV)\oplus Z$.
The result follows because $((u,\omega),A_\scV\tau)_\Omega=(u,\div\tau)_\Omega+(\omega,\tau)_\Omega$.
%
\end{proof}

Finally, we can proceed to proving the main result, which includes Theorem \ref{thm:mutuallywellposed}.

\begin{theorem}
The variational formulations $(\scS)$, $(\scU)$, $(\scD)$, $(\scM)$, $(\scP)$, $(\scS_\Sym)$, $(\scU_\Sym)$ and $(\scM_\Sym)$ are mutually ill or well-posed.
That is, if any single formulation is well-posed, then all others are also well-posed.
\end{theorem}

\begin{proof}
Assume one of the variational formulations is well-posed.
Then by Lemma \ref{lem:Gammazeronotempty}, $\Gamma_0\neq\varnothing$.
Using Lemma \ref{lem:trivialcompatibility}, it follows that for all formulations the compatibility space is trivial so the compatibility conditions are satisfied immediately for any linear form.

It remains to show that the positive inf-sup constants exist for the remaining formulations.
This is proved according to the following implication diagram.
\vspace{-2mm}
\begin{center}
  \begin{tikzpicture}[>=latex]
    \node (P0) at (90:1.7cm) {$(\scM)$};
    \node (P1) at (90+72:1.7cm) {$(\scS)$};
    \node (P2) at (90+2*72:1.7cm) {$(\scP)$};
    \node (P3) at (90+3*72:1.7cm) {$(\scD)$};
    \node (P4) at (90+4*72:1.7cm) {$(\scU)$};
    \node (P5) at (210:0.8cm) {$(\scS_\Sym)$};
    \node (P6) at (330:0.8cm) {$(\scU_\Sym)$};
    \node (P7) at (90:0.8cm) {$(\scM_\Sym)$};
    \draw
    (P0) edge[-implies,double equal sign distance] (P1)
    (P2) edge[-implies,double equal sign distance] (P1)
    (P3) edge[-implies,double equal sign distance] (P2)
    (P4) edge[-implies,double equal sign distance] (P3)
    (P4) edge[-implies,double equal sign distance] (P0)
    (P1) edge[-implies,double equal sign distance] (P5)
    (P5) edge[-implies,double equal sign distance] (P6)
    (P6) edge[-implies,double equal sign distance] (P4)
    (P6) edge[-implies,double equal sign distance] (P7)
    (P7) edge[-implies,double equal sign distance] (P5);
  \end{tikzpicture}
\end{center}
\vspace{-2mm}

$(\scS_\Sym)\Rightarrow(\scU_\Sym)$:
This is the content of Lemma \ref{lem:closedrangesymmetric}.

$(\scS)\Rightarrow(\scS_\Sym)$:
The inf-sup constant $\gamma^{\scS}>0$ is assumed to exist.
Let $x=(\sigma,u)\in X^{\scS_\Sym}\subseteq X^\scS$, $y=(\tau,v)\in Y^{\scS_\Sym}$ and $\tilde{y}=(y,w)\in Y^\scS$, so that $\|y\|_{Y^{\scS_\Sym}}\leq\|\tilde{y}\|_{Y^\scS}$.
Due to the symmetry of $\sigma$ it follows $b_0^{\scS_\Sym}(x,y)=b_0^{\scS}(x,\tilde{y})$.
Hence,
\begin{equation*}
	\gamma^{\scS}\|x\|_{X^{\scS_\Sym}}=\gamma^{\scS}\|x\|_{X^\scS}
		\leq\sup_{\tilde{y}\in Y^\scS\setminus\{0\}}\frac{|b_0^\scS(x,\tilde{y})|}{\|\tilde{y}\|_{Y^\scS}}
			=\sup_{\tilde{y}\in Y^\scS\setminus\{0\}}\frac{|b_0^{\scS_\Sym}(x,y)|}{\|\tilde{y}\|_{Y^\scS}}
				\leq\sup_{y\in Y^{\scS_\Sym}\setminus\{0\}}\frac{|b_0^{\scS_\Sym}(x,y)|}{\|y\|_{Y^{\scS_\Sym}}}\,,
\end{equation*}
so that the desired inf-sup constant $\gamma^{\scS_\Sym}=\gamma^{\scS}>0$ exists and $(\scS_\Sym)$ is well-posed.

$(\scU_\Sym)\Rightarrow(\scU)$:
Let the inf-sup constant $\gamma^{\scU_\Sym}>0$ exist.
Let $x=(\sigma,u)\in X^{\scU_\Sym}$, $\tilde{x}=(x,\omega)\in X^\scU$, $y_\Sym=(\tau_\Sym,v)\in Y^{\scU_\Sym}\subseteq Y^\scU$ and $y=(\tau,v)\in Y^\scU$.
Clearly, $b_0^{\scU_\Sym}(x,y_\Sym)=b_0^{\scU}(\tilde{x},y_\Sym)$, and
\begin{equation*}
	\gamma^{\scU_\Sym}\|x\|_{X^{\scU_\Sym}}
		\leq\sup_{y_\Sym\in Y^{\scU_\Sym}\setminus\{0\}}\frac{|b_0^{\scU_\Sym}(x,y_\Sym)|}{\|y_\Sym\|_{Y^{\scU_\Sym}}}
			=\sup_{y_\Sym\in Y^{\scU_\Sym}\setminus\{0\}}\frac{|b_0^\scU(\tilde{x},y_\Sym)|}{\|y_\Sym\|_{Y^{\scU_\Sym}}}
				\leq\sup_{y\in Y^\scU\setminus\{0\}}\frac{|b_0^\scU(\tilde{x},y)|}{\|y\|_{Y^\scU}}\,.
\end{equation*}
Due to $\|\tilde{x}\|_{X^\scU}^2=\|x\|_{X^{\scU_\Sym}}^2+\|\omega\|_{\bfL^2(\Omega;\Skw)}^2$, it remains to find a bound for $\|\omega\|_{\bfL^2(\Omega;\Skw)}$.
Let $y_0=(\tau,0)\in Y^\scU$, so that $\|y_0\|_{Y^\scU}=\|\tau\|_{\bfH(\div,\Omega)}\geq\|\tau\|_{\bfL^2(\Omega;\Mat)}$ and $(u,\div\tau)_\Omega+(\omega,\tau)_\Omega=b_0^\scU(\tilde{x},y_0)-(\A:\sigma,\tau)_\Omega$.
Then, by Lemma \ref{lem:brezzicondition} it follows
\begin{equation*}
	\begin{aligned}
		C_B\|\omega\|_{\bfL^2(\Omega;\Skw)}
			&\leq\sup_{\tau\in\bfH_{\Gamma_1}(\div,\Omega)\setminus\{0\}}
				\frac{|(u,\div\tau)_\Omega+(\omega,\tau)_\Omega|}{\|\tau\|_{\bfH(\div,\Omega)}}
					=\sup_{\tau\in\bfH_{\Gamma_1}(\div,\Omega)\setminus\{0\}}
						\frac{|b_0^\scU(\tilde{x},y_0)-(\A:\sigma,\tau)_\Omega|}{\|\tau\|_{\bfH(\div,\Omega)}}\\
			&\leq\sup_{\tau\in\bfH_{\Gamma_1}(\div,\Omega)\setminus\{0\}}\frac{|b_0^\scU(\tilde{x},y_0)|}{\|y_0\|_{Y^\scU}}
				+\sup_{\tau\in\bfL^2(\Omega;\Mat)}\frac{|(\A:\sigma,\tau)_\Omega|}{\|\tau\|_{\bfL^2(\Omega;\Mat)}}\\
			&\leq\sup_{y\in Y^\scU\setminus\{0\}}\frac{|b_0^\scU(\tilde{x},y)|}{\|y\|_{Y^\scU}}
				+\|\A\|\|\sigma\|_{\bfL^2(\Omega;\Sym)}
					\leq\Big(1+\frac{\|\A\|}{\gamma^{\scU_\Sym}}\Big)
						\sup_{y\in Y^\scU\setminus\{0\}}\frac{|b_0^\scU(\tilde{x},y)|}{\|y\|_{Y^\scU}}\,,
	\end{aligned}
\end{equation*}
since $\|\sigma\|_{\bfL^2(\Omega;\Sym)}\leq\|x\|_{X^{\scU_\Sym}}$.
Therefore, the existence of the desired inf-sup constant $\gamma^\scS>0$ defined by $(\gamma^\scU)^{-2}=\frac{1}{C_B^2}(1+\frac{\|\A\|}{\gamma^{\scU_\Sym}})^2+(\frac{1}{\gamma^{\scU_\Sym}})^2$ is ensured.

$(\scU)\Rightarrow(\scM)$:
The inf-sup constant $\gamma^{\scU}>0$ is assumed to exist.
Let $x=(\sigma,u,\omega)\in X^\scM$ and $x_\Sym=(\sigma_\Sym,u,\omega)\in X^\scU$, where $\sigma_\Sym=\frac{1}{2}(\sigma+\sigma^\T)$ and $\sigma_\Skw=\frac{1}{2}(\sigma-\sigma^\T)$.
Since $(\sigma_\Sym,\sigma_\Skw)_\Omega=0$, it follows  $\|x\|_{X^\scM}^2=\|x_\Sym\|_{X^\scU}^2+\|\sigma_\Skw\|_{\bfL^2(\Omega;\Skw)}^2+\|\div\sigma\|_{\bfL^2(\Omega)}^2$.
Let $\tilde{y}_w=(0,0,w)\in Y^\scM$ and $\tilde{y}=(\tau,v,w)\in Y^\scM$ so that $\|\tilde{y}_w\|_{Y^\scM}=\|w\|_{\bfL^2(\Omega;\Skw)}$.
Then, it is clear $b_0^\scM(x,\tilde{y}_w)=(\sigma_\Skw,w)_\Omega$, and
\begin{equation*}
	\|\sigma_\Skw\|_{\bfL^2(\Omega;\Skw)}
		=\sup_{w\in\bfL^2(\Omega;\Skw)\setminus\{0\}}\frac{|(\sigma_\Skw,w)_\Omega|}{\|w\|_{\bfL^2(\Omega;\Skw)}}
			=\sup_{w\in\bfL^2(\Omega;\Skw)\setminus\{0\}}\frac{|b_0^\scM(x,\tilde{y}_w)|}{\|\tilde{y}_w\|_{Y^\scM}}
				\leq\sup_{\tilde{y}\in Y^\scM\setminus\{0\}}\frac{|b_0^\scM(x,\tilde{y})|}{\|\tilde{y}\|_{Y^\scM}}\,.
\end{equation*}
Next, let $y=(\tau,v)\in Y^\scU$ and $\tilde{y}_0=(y,0)\in Y^\scM$, so that $\|\tilde{y}_0\|_{Y^\scM}\leq\|y\|_{Y^\scU}$ and $\|v\|_{\bfH^1(\Omega)}\leq\|y\|_{Y^\scU}$.
The distributional identity $-(\div\sigma,v)_\Omega=(\sigma,\nabla v)_\Omega$ holds because 
$\sigma\in\bfH_{\Gamma_1}(\div,\Omega)$ and $v\in\bfH^1_{\Gamma_0}(\Omega)$.
A careful calculation shows that $b_0^{\scM}(x,\tilde{y}_0)=b_0^{\scU}(x_\Sym,y)+(\sigma_\Skw,\nabla v)_\Omega$.
Therefore,
\begin{equation*}
	\begin{aligned}
		\gamma^\scU\|x_\Sym\|_{X^\scU}&\leq\sup_{y\in Y^\scU\setminus\{0\}}\frac{|b_0^\scU(x_\Sym,y)|}{\|y\|_{Y^\scU}}
			\leq\sup_{\tilde{y}_0\in (Y^\scU\setminus\{0\})\times\{0\}}\frac{|b_0^\scM(x,\tilde{y}_0)|}{\|\tilde{y}_0\|_{Y^\scM}}
				+\sup_{v\in\bfH_{\Gamma_0}^1(\Omega)\setminus\{0\}}\frac{|(\sigma_\Skw,\nabla v)_\Omega|}{\|v\|_{\bfH^1(\Omega)}}\\
		&\leq\sup_{\tilde{y}\in Y^\scM\setminus\{0\}}\frac{|b_0^\scM(x,\tilde{y})|}{\|\tilde{y}\|_{Y^\scM}}
			+\|\sigma_\Skw\|_{\bfL^2(\Omega;\Skw)}
				\sup_{v\in\bfH_{\Gamma_0}^1(\Omega)\setminus\{0\}}\frac{\|\nabla v\|_{\bfL^2(\Omega)}}{\|v\|_{\bfH^1(\Omega)}}
		\leq 2\sup_{\tilde{y}\in Y^\scM\setminus\{0\}}\frac{|b_0^\scM(x,\tilde{y})|}{\|\tilde{y}\|_{Y^\scM}}\,.
	\end{aligned}
\end{equation*}
Finally, let $\tilde{y}_v=(0,v,0)\in Y^\scM$ so that $\|\tilde{y}_v\|_Y^\scM=\|v\|_{\bfL^2(\Omega)}$ and $-(\div\sigma,v)_\Omega=b_0^\scM(x,\tilde{y}_v)$.
Then,
\begin{equation*}
	\|\div\sigma\|_{\bfL^2(\Omega)}=\sup_{v\in\bfL^2(\Omega)\setminus\{0\}}\frac{|(\div\sigma,v)_\Omega|}{\|v\|_{\bfL^2(\Omega)}}
		=\sup_{v\in\bfL^2(\Omega)\setminus\{0\}}\frac{|b_0^\scM(x,\tilde{y}_v)|}{\|\tilde{y}_v\|_{Y^\scM}}
			\leq\sup_{\tilde{y}\in Y^\scM\setminus\{0\}}\frac{|b_0^\scM(x,\tilde{y})|}{\|\tilde{y}\|_{Y^\scM}}\,,
\end{equation*}
which implies that $\gamma^\scM>0$ defined by $(\gamma^\scM)^2=\frac{(\gamma^\scU)^2}{4+2(\gamma^\scU)^2}$ is the desired inf-sup constant.

$(\scU_\Sym)\Rightarrow(\scM_\Sym)$:
This is proved analogously to $(\scU)\Rightarrow(\scM)$, but ignoring the calculations associated to the term $\sigma_\Skw$, which vanishes in this symmetric setting.

$(\scM)\Rightarrow(\scS)$:
The inf-sup constant of $(\scM)$, $\gamma^{\scM}>0$, is assumed to exist.
Let $x=(\sigma,u)\in X^\scS$, $\tilde{x}=(x,\frac{1}{2}(\nabla u-\nabla u^\T))\in X^\scM$, $y_\scM=(\tau,v,w)\in Y^\scM$ and $y_\scS=(\A:\tau,v,w)\in Y^\scS$. 
Then, notice that $\|x\|_{X^\scS}^2=\|\tilde{x}\|_{X^\scM}^2+\|\varepsilon(u)\|_{\bfL^2(\Omega;\Sym)}^2$ and $\|\A:\tau\|_{\bfL^2(\Omega;\Sym)}\leq\|\A\|\|\tau\|_{\bfL^2(\Omega;\Mat)}$, so $\|y_\scS\|_{Y^\scS}\leq M_\A\|y_\scM\|_{Y^\scM}$ where $M_\A=\max\{\|\A\|,1\}$.
The distributional identity $(u,\div\tau)_\Omega=-(\nabla u,\tau)_\Omega$ holds because $u\in\bfH^1_{\Gamma_0}(\Omega)$ and $\tau\in\bfH_{\Gamma_1}(\div,\Omega)$, and implies that $b_0^{\scM}(\tilde{x},y_\scM)=b_0^{\scS}(x,y_\scS)$.
Hence, 
\begin{equation*}
	\begin{aligned}
		\gamma^\scM\|\tilde{x}\|_{X^\scM}&\leq\sup_{y_\scM\in Y^\scM\setminus\{0\}}\frac{|b_0^\scM(\tilde{x},y_\scM)|}{\|y_\scM\|_{Y^\scM}}
			\leq M_\A\sup_{y_\scM\in Y^\scM,y_\scS\neq0}\frac{|b_0^\scS(x,y_\scS)|}{\|y_\scS\|_{Y^\scS}}
				\leq M_\A\sup_{y\in Y^\scS\setminus\{0\}}\frac{|b_0^\scS(x,y)|}{\|y\|_{Y^\scS}}\,.
	\end{aligned}
\end{equation*}
It remains to find a bound for $\|\varepsilon(u)\|_{\bfL^2(\Omega;\Sym)}$.
Let $y_0=(\A:\tau_\Sym,0,0)\in Y^\scS$ for $\tau_\Sym\in\bfL^2(\Omega;\Sym)$, so that $\|y_0\|_{Y^\scS}\leq\|\A\|\|\tau_\Sym\|_{\bfL^2(\Omega;\Sym)}$.
Notice $(\varepsilon(u),\tau_\Sym)_\Omega=(\sigma,\A:\tau_\Sym)_\Omega-b_0^\scS(x,y_0)$.
Therefore,
\begin{equation*}
	\begin{aligned}
	\|\varepsilon(u)\|_{\bfL^2(\Omega;\Sym)}
		&=\sup_{\tau_\Sym\in\bfL^2(\Omega;\Sym)\setminus\{0\}}
			\frac{|(\varepsilon(u),\tau_\Sym)_\Omega|}{\|\tau_\Sym\|_{\bfL^2(\Omega;\Sym)}}
				=\sup_{\tau_\Sym\in\bfL^2(\Omega;\Sym)\setminus\{0\}}
					\frac{|b_0^\scS(x,y_0)-(\sigma,\A:\tau_\Sym)_\Omega|}{\|\tau_\Sym\|_{\bfL^2(\Omega;\Sym)}}\\
		&\leq\|\A\|\sup_{\tau_\Sym\in\bfL^2(\Omega;\Sym)\setminus\{0\}}\frac{|b_0^\scS(x,y_0)|}{\|y_0\|_{Y^\scS}}
			+\|\A\|\sup_{\tau_\Sym\in\bfL^2(\Omega;\Sym)\setminus\{0\}}
				\frac{|(\sigma,\A:\tau_\Sym)_\Omega|}{\|\A:\tau_\Sym\|_{\bfL^2(\Omega;\Sym)}}\\
		&\leq\|\A\|\sup_{y\in Y^\scS\setminus\{0\}}\frac{|b_0^\scS(x,y)|}{\|y\|_{Y^\scS}}+\|\A\|\|\sigma\|_{\bfL^2(\Omega;\Sym)}\,,					
	\end{aligned}
\end{equation*}
where it is used that $\A$ is bijective on $\bfL^2(\Omega;\Sym)$.
Using that $\|\sigma\|_{\bfL^2(\Omega;\Sym)}\leq\|\tilde{x}\|_{X^\scM}$, the existence of the inf-sup constant $\gamma^\scS>0$ defined by $(\gamma^\scS)^{-2}=\|\A\|^2(1+\frac{M_\A}{\gamma^\scM})^2+(\frac{M_\A}{\gamma^\scM})^2$ is ensured.

$(\scM_\Sym)\Rightarrow(\scS_\Sym)$:
This is proved analogously to $(\scM)\Rightarrow(\scS)$.

$(\scU)\Rightarrow(\scD)$:
The inf-sup constant of $(\scU)$, $\gamma^{\scU}>0$, is assumed to exist.
Let $x=(\sigma,u)\in X^\scD$, $\tilde{x}=(x,\frac{1}{2}(\nabla u-\nabla u^\T))\in X^\scU$, $y_\scU=(\tau,v)\in Y^\scU$, $y_\scD=(\A:\tau,v)\in Y^\scD$, and $y_0=(\A:\tau_\Sym,0)\in Y^\scD$ for $\tau_\Sym\in\bfL^2(\Omega;\Sym)$.
The proof is then the same as that for $(\scM)\Rightarrow(\scS)$, but replacing $\scM$ by $\scU$ and $\scS$ by $\scD$.

$(\scD)\Rightarrow(\scP)$:
Assume the inf-sup constant $\gamma^{\scD}>0$ exists.
Let $x=u\in X^\scP$, $\tilde{x}=(\C:\nabla u,x)\in X^\scD$, $y=v\in Y^\scP$ and $\tilde{y}=(\tau,y)\in Y^\scD$, so that $\|x\|_{X^\scP}\leq\|\tilde{x}\|_{X^\scD}$ and $\|y\|_{Y^\scP}\leq\|\tilde{y}\|_{Y^\scD}$.
Clearly it holds that $b_0^\scD(\tilde{x},\tilde{y})=b_0^\scP(x,y)$.
Then,
\begin{equation*}
	\gamma^\scD\|x\|_{X^\scP}\leq\gamma^{\scD}\|\tilde{x}\|_{X^\scD}
		\leq\sup_{\tilde{y}\in Y^\scD\setminus\{0\}}\frac{|b_0^\scD(\tilde{x},\tilde{y})|}{\|\tilde{y}\|_{Y^\scD}}
			=\sup_{\tilde{y}\in Y^\scD\setminus\{0\}}\frac{|b_0^\scP(x,y)|}{\|\tilde{y}\|_{Y^\scD}}
				\leq\sup_{y\in Y^\scP\setminus\{0\}}\frac{|b_0^\scP(x,y)|}{\|y\|_{Y^\scP}}\,,
\end{equation*}
so that the desired inf-sup constant $\gamma^\scP=\gamma^\scD>0$ exists and $(\scP)$ is well-posed.

$(\scP)\Rightarrow(\scS)$:
The inf-sup constant $\gamma^{\scP}>0$ is assumed to exist.
Let $x=u\in X^\scP$, $\tilde{x}=(\sigma,x)\in X^\scS$, $y=v\in Y^\scP$ and $\tilde{y}_v=(-\varepsilon(v),v,\frac{1}{2}(\nabla v-\nabla v^\T))\in Y^\scS$.
Then, notice that $\|y\|_{Y^\scP}=\|\tilde{y}_v\|_{Y^\scS}$ and that $\|\tilde{x}\|_{X^\scS}^2=\|x\|_{X^\scP}^2+\|\sigma\|_{\bfH(\div,\Omega)}^2$.
A careful calculation yields $b_0^\scS(\tilde{x},\tilde{y}_v)=b_0^\scP(x,y)$.
Therefore,
\begin{equation*}
	\gamma^{\scP}\|x\|_{X^\scP}\leq\sup_{y\in Y^\scP\setminus\{0\}}\frac{|b_0^\scP(x,y)|}{\|y\|_{Y^\scP}}
		=\sup_{y\in Y^\scP\setminus\{0\}}\frac{|b_0^\scS(\tilde{x},\tilde{y}_v)|}{\|\tilde{y}_v\|_{Y^\scS}}
			\leq\sup_{\tilde{y}\in Y^\scS\setminus\{0\}}\frac{|b_0^\scS(\tilde{x},\tilde{y})|}{\|\tilde{y}\|_{Y^\scS}}\,.
\end{equation*}
Next, consider $\xi\in\bfL^2(\Omega;\Mat)$ which is decomposed into $\xi_\Sym=\frac{1}{2}(\xi+\xi^\T)$ and $\xi_\Skw=\frac{1}{2}(\xi-\xi^\T)$, and let $\tilde{y}_\xi=(\xi_\Sym,0,\xi_\Skw)$, so that $\|\xi\|_{\bfL^2(\Omega;\Mat)}=\|\tilde{y}_\xi\|_{Y^\scS}$.
Notice that $b_0^\scS(\tilde{x},\tilde{y}_\xi)=(\sigma,\xi)_\Omega-(\C:\nabla u,\xi_\Sym)_\Omega$.
Hence,
\begin{equation*}
	\begin{aligned}
		\|\sigma\|_{\bfL^2(\Omega;\Mat)}
			&=\sup_{\xi\in\bfL^2(\Omega;\Mat)\setminus\{0\}}\frac{|(\sigma,\xi)_\Omega|}{\|\xi\|_{\bfL^2(\Omega;\Mat)}}
				=\sup_{\xi\in\bfL^2(\Omega;\Mat)\setminus\{0\}}
					\frac{|b_0^\scS(\tilde{x},\tilde{y}_\xi)+(\C:\nabla u,\xi_\Sym)_\Omega|}{\|\tilde{y}_\xi\|_{Y^\scS}}\\
		&\leq\sup_{\xi\in\bfL^2(\Omega;\Mat)\setminus\{0\}}\frac{|b_0^\scS(\tilde{x},\tilde{y}_\xi)|}{\|\tilde{y}_\xi\|_{Y^\scS}}
			+\sup_{\xi_\Sym\in\bfL^2(\Omega;\Sym)\setminus\{0\}}\frac{|(\C:\nabla u,\xi_\Sym)_\Omega|}{\|\xi_\Sym\|_{\bfL^2(\Omega;\Sym)}}\\
		&\leq\sup_{\tilde{y}\in Y^\scS\setminus\{0\}}\frac{|b_0^\scS(\tilde{x},\tilde{y})|}{\|\tilde{y}\|_{Y^\scS}}
			+\|\C:\nabla u\|_{\bfL^2(\Omega;\Sym)}
				\leq\sup_{\tilde{y}\in Y^\scS\setminus\{0\}}\frac{|b_0^\scS(\tilde{x},\tilde{y})|}{\|\tilde{y}\|_{Y^\scS}}
					+\|\C\|\|x\|_{X^\scP}\,.			
	\end{aligned}			
\end{equation*}
Finally, let $\tilde{y}_0=(0,v,0)\in Y^\scS$ so that $\|\tilde{y}_0\|_Y^\scS=\|v\|_{\bfL^2(\Omega)}$ and $-(\div\sigma,v)_\Omega=b_0^\scS(\tilde{x},\tilde{y}_0)$.
Then,
\begin{equation*}
	\|\div\sigma\|_{\bfL^2(\Omega)}=\sup_{v\in\bfL^2(\Omega)\setminus\{0\}}\frac{|(\div\sigma,v)_\Omega|}{\|v\|_{\bfL^2(\Omega)}}
		=\sup_{v\in\bfL^2(\Omega)\setminus\{0\}}\frac{|b_0^\scS(\tilde{x},\tilde{y}_0)|}{\|\tilde{y}_0\|_{Y^\scS}}
			\leq\sup_{\tilde{y}\in Y^\scS\setminus\{0\}}\frac{|b_0^\scS(\tilde{x},\tilde{y})|}{\|\tilde{y}\|_{Y^\scS}}\,.
\end{equation*}
Therefore, the inf-sup constant $\gamma^\scS>0$ exists and is defined by $(\gamma^\scS)^{-2}=(1+\frac{\|\C\|}{\gamma^\scP})^2+(\frac{1}{\gamma^\scP})^2+1$.
\end{proof}

\end{document}